\documentclass[runningheads]{llncs}
\usepackage[T1]{fontenc}
\usepackage{graphicx}
\usepackage{tikz}
\usepackage{xcolor}
\usepackage[ruled,vlined, linesnumbered]{algorithm2e}
\usepackage{xspace}
\usepackage{amssymb}
\usepackage{amsmath}
\usepackage{placeins}
\usepackage{subcaption}

\definecolor{safeA1}{RGB}{26,133,255}  
\definecolor{safeA2}{RGB}{212,17,89}   

\newcommand{\Pgen}{\ensuremath{\mathcal{P}}\xspace}

\newcommand{\Pint}{\ensuremath{\mathcal{P}_{IP}}\xspace}

\newcommand{\Pex}{\ensuremath{\mathcal{P}_{ex}}\xspace}
\newcommand{\solP}{\ensuremath{Sol(\mathcal{P})}\xspace}
\newcommand{\solPint}{\ensuremath{Sol(\mathcal{P_{IP}})}\xspace}

\newcommand{\width}{\ensuremath{w}\xspace}
\newcommand{\genDD}{\ensuremath{\mathcal{M}}\xspace}
\newcommand{\resDD}{\ensuremath{\mathcal{M}^-}\xspace}
\newcommand{\relDD}{\ensuremath{\mathcal{M}^+}\xspace}

\begin{document}
\title{Implicit Decision Diagrams}
%
%
\author{Isaac Rudich\inst{1, 2}\orcidID{0000-0002-3106-1020} \and
Louis-Martin Rousseau\inst{1}\orcidID{0000-0001-6949-6014}}
\authorrunning{I. Rudich and L.M. Rousseau}
%
\institute{Polytechnique Montréal, Quebec, Canada \and
Carnegie Mellon University, Pittsburgh PA, USA}
\maketitle              
\begin{abstract}
Decision Diagrams (DDs) have emerged as a powerful tool for discrete optimization, with rapidly growing adoption.
DDs are directed acyclic layered graphs; restricted DDs are a generalized greedy heuristic for finding feasible solutions, and relaxed DDs compute combinatorial relaxed bounds.
There is substantial theory that leverages DD-based bounding, yet the complexity of constructing the DDs themselves has received little attention.
Standard restricted DD construction requires $\mathcal{O}(\width \log \width)$ per layer; standard relaxed DD construction requires $\mathcal{O}(\width^2)$, where $\width$ is the width of the DD.
Increasing $\width$ improves bound quality at the cost of more time and memory.

We introduce \emph{implicit Decision Diagrams}, storing arcs implicitly rather than explicitly, and reducing per-layer complexity to $\mathcal{O}(\width)$ for restricted and relaxed DDs.
We prove this is optimal: any framework treating state-update and merge operations as black boxes cannot do better.

Optimal complexity shifts the challenge from algorithmic overhead to low-level engineering.
We show how implicit DDs can drive a MIP solver, and release \texttt{ImplicitDDs.jl}, an open-source Julia solver exploiting the implementation refinements our theory enables.
Experiments demonstrate the solver outperforms Gurobi on Subset Sum.

\end{abstract}

\section{Introduction}

The remarkable success of Mixed-Integer Programming (MIP) solvers stems not just from theoretical and algorithmic advances, but from highly optimized implementations \cite{koch2022progress}.
Solvers like Gurobi~\cite{gurobi}, SCIP~\cite{achterberg2009}, and HiGHS~\cite{huangfu2018parallelizing} succeed in part through low-level optimization, transforming theory into practice.

While MIP solvers have benefited from decades of theoretical and engineering refinement, Decision Diagrams (DDs) \cite{BnB,BnB2} have not.
DDs represent discrete optimization problems as acyclic directed layered graphs, where each root-to-terminal path encodes a potential solution.
Since 2016 \cite{DDforO}, DD-based techniques have seen rapid adoption, with a 2022 survey reviewing over 100 papers \cite{castro2022decision}.
DD-based solvers include ddo \cite{gillard2021ddo,gillard2022discrete,coppe2024decision,ddo}, CODD \cite{michel2024codd}, Peel-and-Bound \cite{Rudich2022PeelAndBoundGS,Rudich2023ImprovedPnB}, DIDP \cite{kuroiwa2023domain,beck_et_al:LIPIcs.CP.2025.5}, and HADDOCK \cite{gentzel2020haddock,gentzel2023optimization}, each competitive on specific problem classes.
In general, DDs are often effective for problems that can be formulated more compactly as Dynamic Programs than as MIPs \cite{MDDForSP}. 

Despite growing adoption, DD-based solvers remain in their infancy compared to MIP solvers.
There is substantial theory that leverages DD-based bounding techniques, yet the complexity of constructing the underlying DDs themselves has received little attention.
Addressing this matters: optimal algorithmic complexity shifts the challenge from algorithmic design to low-level engineering.
Restricted DDs provide primal bounds by encoding feasible solutions; standard construction requires $\mathcal{O}(\width \log \width)$ per layer (where $\width$ is the width of the DD).
Relaxed DDs provide dual bounds by over-approximating the feasible region; standard construction requires $\mathcal{O}(\width^2)$ per layer.
This dependency on $\width$ is critical, because in both cases bound quality improves with $\width$.

This paper makes five contributions to the theory of DDs.
(1) We develop a framework for analyzing DD construction complexity that separates the inherent costs from problem-specific costs.
Inherent costs are the costs incurred by DD construction, independent of the problem being solved. 
(2) We design \emph{implicit Decision Diagrams}, where arcs are stored implicitly, reducing per-layer complexity from $\mathcal{O}(\width \log \width)$ to $\mathcal{O}(\width)$ for restricted DDs and from $\mathcal{O}(\width^2)$ to $\mathcal{O}(\width)$ for relaxed DDs.
(3) We prove this construction is optimal: any general framework treating state-update and merge as black boxes cannot beat the $\mathcal{O}(\width)$ factor.
(4) For integer programming (IP), we show that problem-specific costs are subsumed by inherent costs.
(5) We show how implicit DDs can drive a MIP solver.

We also contribute to DD practice.
Having addressed the problem of optimal DD construction, we turn to low-level engineering.
We release \texttt{ImplicitDDs.jl}, a highly optimized open-source MIP solver in Julia for problems with bounded variables.
The solver accepts models via JuMP, Julia's standard optimization framework.
We validate our theoretical claims experimentally, and demonstrate the solver outperforms Gurobi on Subset Sum.

The paper proceeds as follows.
Section~\ref{sec:ip_notation} and ~\ref{sec:strategy} provide general framing.
Section~\ref{sec:standard_dds} analyzes standard DDs.
Section~\ref{sec:implicit_dds} presents implicit DDs.
Section~\ref{sec:mip} addresses MIP.
Section~\ref{sec:experiments} provides experiments, and Section~\ref{sec:conclusion} concludes.

\section{Integer Programming Notation}
\label{sec:ip_notation}

Our framework is general-purpose; it can be used for Constraint Programming as well as IP.
We use IP as our running example because its problem-specific costs are subsumed by inherent costs, providing a clean baseline where only inherent costs matter.
Section~\ref{sec:mip} extends the approach to MIP.

We assume all IPs have objectives normalized for minimization, and constraints normalized for $\leq$, matching our solver and framework.
Let $n$ be the number of integer variables and $m$ be the number of constraints.
Let $A \in \mathbb{R}^{m \times n}$ be the constraint coefficient matrix with entries $a_{ij}$. Let $c \in \mathbb{R}^n$ be the objective coefficients, and $b \in \mathbb{R}^m$ be the right-hand side (RHS) values.
Let $\mathit{lb}, \mathit{ub} \in \mathbb{Z}^n$ be the lower and upper bounds on the variables, respectively.
Our approach requires all variables to have finite bounds. Thus:
\begin{align*}
  \min \quad & c^\top x \\
  \text{s.t.} \quad & Ax \leq b \\
  & \mathit{lb} \leq x \leq \mathit{ub} \\
  & x \in \mathbb{Z}^n
\end{align*}

Let \Pint be an IP with $n$ variables $\{x_1, \ldots, x_n\}$; \solPint the set of feasible solutions; $x^*$ an optimal solution with value $z^*(\Pint)$; and $d(x_i)$ the domain of $x_i$.
Let $K = \max_{i \in \{1, \ldots, n\}} |d(x_i)|$ be the size of the largest domain.

\subsection{Example}
\label{subsec:ipex}

We use the following example \Pex throughout: 
\begin{align*}
\min\quad & -3x_1 - 2x_2 - 2x_3 - x_4 \\
\text{s.t.}\quad & x_1 + x_4 \le 1\\
& x_2 + x_3 \le 1\\
& x_i \in \{0,1\}, \quad i=1,\dots,4
\end{align*}

\subsection{Monotonic Constraint Residuals}
Our IP-specific approach uses \emph{constraint residuals} that are normalized to be monotonic, inspired by bound strengthening in MIP presolve \cite{achterberg2018}. 
Given a constraint $i \in \{1, \ldots, m\}$, let the \emph{monotonic constraint residual} $r_i$ be the remaining slack in $i$ when all unfixed variables minimize their contribution to the left hand side (LHS) of $i$.
For $k$ fixed variables $\{x_1, \ldots, x_k\}$, the residual for $i$ is:
\begin{equation}
\label{eq:residuals}
r_i = b_i - \sum_{j=1}^{k} a_{ij}x_j - \sum_{j=k+1}^{n} a_{ij} \cdot
\begin{cases}
\mathit{lb}_j & \text{if } a_{ij} > 0\\
\mathit{ub}_j & \text{if } a_{ij} < 0
\end{cases}
\end{equation}

The \emph{initial residuals} ($k=0$) are:
\begin{equation}
\label{eq:initial_residuals}
r_i^{(0)} = b_i - \sum_{j=1}^{n} a_{ij} \cdot
\begin{cases}
\mathit{lb}_j & \text{if } a_{ij} > 0\\
\mathit{ub}_j & \text{if } a_{ij} < 0
\end{cases}
\end{equation}

Assigning $x_k = v$ updates residuals recursively:
\begin{equation}
\label{eq:residual_update}
r_i^{(k)} = r_i^{(k-1)} + a_{ik} \cdot
\begin{cases}
\mathit{lb}_k - v & \text{if } a_{ik} > 0\\
\mathit{ub}_k - v & \text{if } a_{ik} < 0
\end{cases}
\end{equation}

Residuals can only decrease (or stay constant) as variables are fixed. A negative residual ($r_i < 0$) indicates constraint $i$ is violated. Henceforth, we refer to monotonic constraint residuals simply as residuals. 
The initial residuals for \Pex are $\mathbf{r}^{(0)} = [1, 1]$, since all variables contribute 0 at their minima. Given the partial solution $x_1 = 1$, the residuals become $\mathbf{r} = [0, 1]$.

\section{Complexity Analysis Strategy}
\label{sec:strategy}
Our general framework for constructing DDs includes problem-specific functions. 
To separate inherent costs from problem-specific costs, we treat these functions as black boxes.
However, there is a minimum cost to using these functions that is inherent to DD construction.
We explicitly identify those as inherent costs.
For example, $\textsc{UpdateState}$ must read and write a state, so calling it incurs an inherent cost proportional to the state size.

\section{Standard Decision Diagrams}
\label{sec:standard_dds}

This section reviews \cite{DDforO}; the complexity analysis is original.

\subsection{Exact Decision Diagrams}
\label{sec:exactdds}
A DD for a combinatorial optimization problem \Pgen is a directed layered acyclic graph with root $r$ and terminal $t$. 
An \emph{exact DD} represents \Pgen if every path from $r$ to $t$ encodes a feasible solution to \Pgen, and every feasible solution to \Pgen is encoded as a path from $r$ to $t$. 
A DD is exact for \Pgen iff the encoded solutions equal \solP.

We now construct an exact DD for \Pex using an approach similar to a decision tree.
Initialize root $r$ on layer 0 with initial residuals $\mathbf{r} = [1, 1]$ and objective $0$.
For each layer $i = 0, \ldots, n-1$: 
(1) for each node in the layer, create a child node for each value in $d(x_{i+1})$,
(2) update objective values and residuals,
(3) merge nodes with identical residuals, keeping only the minimal objective value, and
(4) remove infeasible nodes (negative residuals).
Merging preserves optimality because identical residuals represent equivalent constraint states.
Finally, feasible nodes on layer $n$ merge into terminal $t$ (state is now irrelevant). 
The shortest $r$-$t$ path is optimal (Figure~\ref{fig:exact_dd}).
This is exponential, and thus generally intractable.

\begin{figure}[h]
\centering
\includegraphics[width = \textwidth]{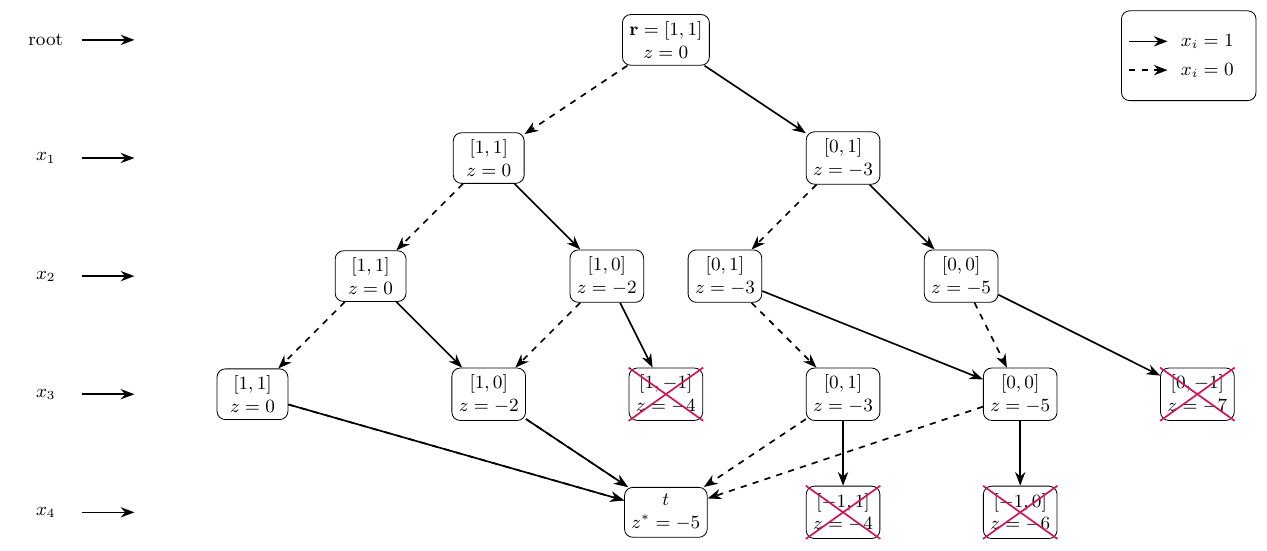}
\caption{Exact decision diagram for \Pex. \textcolor{safeA2}{X} indicates a constraint violation.}
\label{fig:exact_dd}
\end{figure}

\subsection{Restricted Decision Diagrams}
\label{sec:restricted_dd}

A \emph{restricted DD} \resDD for \Pgen is a generalized greedy heuristic for finding feasible solutions.
Unlike exact DDs, restricted DDs encode only a subset of feasible solutions \cite{DDforO}.
Let $Sol(\genDD)$ be the solutions encoded by \genDD. For a restricted DD: $Sol(\resDD) \subseteq \solP$.
The optimal path through \resDD (minimum-cost for minimization) is a feasible solution, providing an upper bound for \Pgen.

Restricted DD construction mirrors exact construction, except each layer is limited to width \width by heuristically removing nodes.
Let $\ell(u)$ be the layer of node $u$, $d(u)$ be the domain of the next variable $x_{\ell(u) + 1}$, $s(u)$ be the state vector of $u$ (which includes the objective value at $u$, $f(u)$, as well as other problem-specific information), and $|s(u)|$ be the length of $s(u)$.
Algorithm \ref{algo:topdown_restricted} formalizes standard restricted DD construction, and Figure \ref{fig:restricted_dd} shows a restricted DD for \Pex. 

\begin{algorithm}[h!t]
\SetAlgoLined
    \textbf{Input:} The root node $r$ for a given \Pgen. Problem-specific functions:  $\textsc{UpdateState}$ and $\textsc{Mergeable}$. \\
    Initialize $Q \leftarrow \{r\}$, $Q_{next} \leftarrow \emptyset$\\
   
    \While{$Q \neq \emptyset$}{
        $Q_{next} \leftarrow \emptyset$\\
        \tcp{1. Child generation}
        \ForEach{node $u \in Q$}{
            \ForEach{label $l \in d(u)$}{
                Create a new node $v$: $Q_{next} \leftarrow Q_{next} \cup \{v\}$\\
                Add arc $a_{uv}$, labeled $l$, from $u$ to $v$; $s(v) = \textup{\textsc{UpdateState}}(s(u), x_{\ell(u)+1}, l)$\\
            }
        }
        \tcp{2. Merging}
        \textbf{(Optional)} For each pair of nodes where \textup{\textsc{Mergeable}}($s(u)$, $s(v)$) $=$ true, keep only the node with minimum $f(\cdot)$\\
        \tcp{3. Feasibility pruning}
        \ForEach{node $u \in Q_{next}$}{
            \If{\textup{\textsc{IsFeasible}}($s(u)$) $=$ false}{
                $Q_{next} \leftarrow Q_{next} \backslash \{u\}$\\
            }
        }
        \tcp{4. Width limiting}
        \If{$|Q_{next}| > \width$}{
            Sort $ u \in Q_{next}$ by increasing objective value ($f(u)$)\\
            Trim $Q_{next}$ to length \width by removing the last elements\\
        }
        \tcp{5. Additional operations}
        \textbf{(Optional)} Apply additional operations (e.g., rough bounding \cite{ddo})\\

        $Q \leftarrow Q_{next}$\\
    }
    
    \Return{\resDD}
    \caption{Standard Restricted Decision Diagrams \cite{DDforO}}
    \label{algo:topdown_restricted}
\end{algorithm}

\begin{figure}[h!t]
\centering
\includegraphics[width = \textwidth]{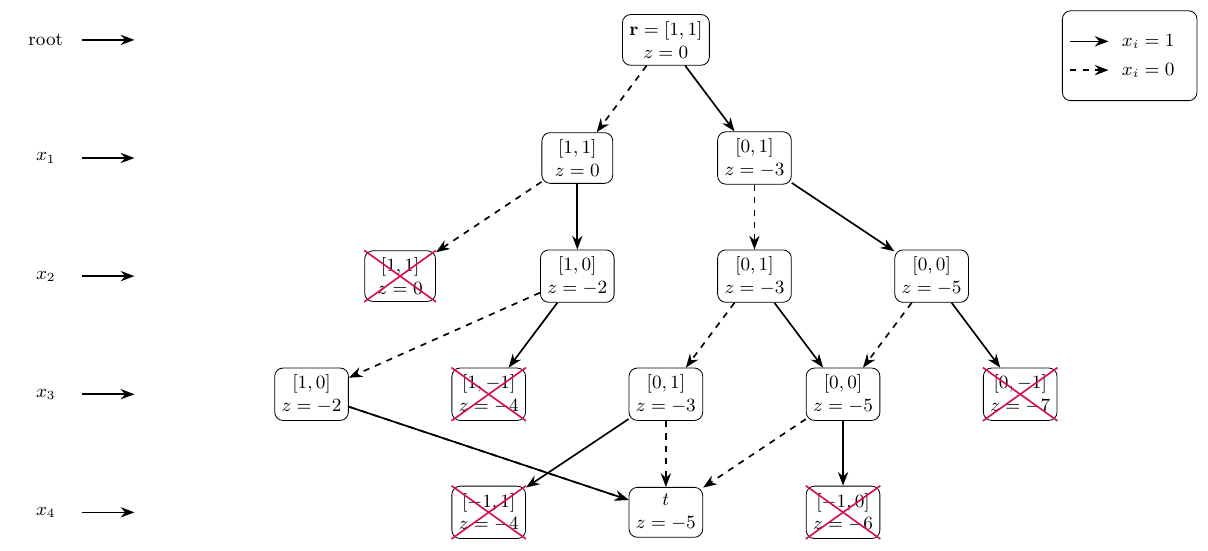}
\caption{Restricted DD for \Pex, $\width = 3$. \textcolor{safeA2}{X} marks infeasible or trimmed nodes.}
\label{fig:restricted_dd}
\end{figure}

\vspace{-3em}

\subsubsection{Generalized Complexity Analysis for Restricted DDs}
We analyze Algorithm \ref{algo:topdown_restricted}, which performs five operations per layer: (1) child generation, (2) merging, (3) feasibility pruning, (4) width limiting, and (5) other. 
Recall that $K$ is the size of the largest domain of any variable, and $|s|$ is the length of the state vector being stored at each node. The \textcolor{safeA1}{highlighted} term is the inerent cost.

\textbf{Child generation:} Creates up to $K \cdot \width$ new nodes by calling $\textsc{UpdateState}$ for each node-label pair, requiring $\mathcal{O}(|s|) + C_{updt}$ per call, where $C_{updt}$ is the problem-specific cost of $\textsc{UpdateState}$.
Per layer: \[
    \textcolor{safeA1}{\mathcal{O}(K \cdot \width \cdot |s|)} + \mathcal{O}(K \cdot \width) \cdot C_{updt}
\]

\textbf{Merging without hashing:} Merges equivalent nodes by calling $\textsc{Mergeable}$ on pairs, requiring $\mathcal{O}(|s|) + C_{mrg}$ per comparison, where $C_{mrg}$ is the problem-specific cost of $\textsc{Mergeable}$. 
Finding all equivalent pairs is done by sorting and requires $\mathcal{O}(K \cdot \width \cdot \log(K \cdot \width))$ node comparisons.
Per layer:
\[
 \textcolor{safeA1}{\mathcal{O}(K \cdot \width \cdot |s| \cdot \log(K \cdot \width))} +  \mathcal{O}(K \cdot \width \cdot \log(K \cdot \width)) \cdot C_{mrg}
\]

\textbf{Merging with hashing:} 
For many problems, nodes can be grouped by hashing instead of sorting, requiring $\mathcal{O}(|s|)$ operations per node \cite{kuroiwa2023domain}. Matching hashes define equivalence. This eliminates the problem-specific cost of calling  $\textsc{Mergeable}$.
Per layer: \[
\textcolor{safeA1}{\mathcal{O}(K \cdot \width \cdot |s|)}
\]

\textbf{Feasibility pruning:} 
Calls $\textsc{IsFeasible}$ on each node, requiring $\mathcal{O}(|s|) + C_{isf}$ per node, where $C_{isf}$ is the problem-specific cost of $\textsc{IsFeasible}$.
Per layer: \[
\textcolor{safeA1}{\mathcal{O}(K \cdot \width \cdot |s|)} + \mathcal{O}(K \cdot \width) \cdot C_{isf}
\]

\textbf{Width limiting:} Sorts nodes by objective value and trims to $\width$.
Per layer: \[
\textcolor{safeA1}{\mathcal{O}(K \cdot \width \cdot \log(K \cdot \width))}
\]

\textbf{Additional operations:} Optional operations with cost $C_{othr}$ (excluded).
\vspace{1em}\\
The inherent framework cost is the \textcolor{safeA1}{component} without any problem-specific costs ($C_{updt}$, $C_{mrg}$, $C_{isf}$, $C_{othr}$). 
For $n$ layers:

\textbf{Merging without hashing:}
\[
\textcolor{safeA1}{\mathcal{O}\Big(n \cdot K \cdot \width \cdot |s| \cdot \log(K \cdot \width)\Big)} + \mathcal{O}( n \cdot K \cdot \width) \cdot \Big(C_{updt} +  C_{isf} + \mathcal{O}\big(\log(K \cdot \width)\big) \cdot C_{mrg}\Big)
\]

\textbf{Merging with hashing (or without merging):}
\begin{gather*}
\textcolor{safeA1}{\mathcal{O}(n \cdot K \cdot \width \cdot |s|) + \mathcal{O}(n \cdot K \cdot \width \cdot \log(K \cdot \width))} + \mathcal{O}( n \cdot K \cdot \width) \cdot (C_{updt} +  C_{isf})\\
\textcolor{safeA1}{\mathcal{O}\Big(n \cdot K \cdot \width \cdot \max\big(|s|, \log(K \cdot \width)\big)\Big)} + \mathcal{O}( n \cdot K \cdot \width) \cdot (C_{updt} +  C_{isf})
\end{gather*}

The goal is to maximize \width while maintaining tractability.
When $\log(K \cdot \width) > |s|$, both hashing and no merging achieve $\mathcal{O}(n \cdot K \cdot \width \cdot \log(K \cdot \width))$.
The complexity is dominated by width limiting; scaling to larger \width requires a different approach.

\subsubsection{Application to Integer Programs}
For \Pint, $s(u) = \mathbf{r}(u)$, where $\mathbf{r}(u)$ is the residual vector at $u$, so $|s| = m$.
\textsc{UpdateState} computes residuals using Equation \ref{eq:residual_update}, $\mathcal{O}(m)$.
\textsc{IsFeasible} checks if all residuals are non-negative $\mathcal{O}(m)$.
Residuals can use hash-based merging (no problem-specific cost).
The total problem-specific costs are $\mathcal{O}(m)$.
When $\log(K \cdot \width) > m$, problem-specific costs are dominated by inherent costs; otherwise they are subsumed.

\subsection{Relaxed Decision Diagrams}
\label{sec:relaxed_dd}

A \emph{relaxed DD} \relDD provides a lower (relaxed) bound by encoding all feasible solutions plus infeasible ones: $\solP \subseteq Sol(\relDD)$ \cite{DDforO}.
Two standard approaches are \textit{top-down merging} and \textit{construction by separation}, both in $\mathcal{O}(\width^2)$ time.
Section \ref{subsec:implicit_rel_dds} uses separation, so we present it here; top-down is in Appendix \ref{appendix:topdown_relaxed}.

Construction by separation starts with a $\width$=1 DD and splits nodes.
Three new operations: \textsc{MergeStates}$(s(u), s(v))$, \textsc{SelectNodeToSplit}$(L)$, and \textsc{PartitionArcs}$(u)$.
For arc $a$, let $p(a)$ and $l(a)$ be its parent and label.
Algorithm \ref{algo:separation_relaxed} formalizes this; Figure \ref{fig:relaxed_dd} demonstrates it for $\Pex$.

\begin{figure}[h!t]
\centering
\includegraphics[width = 0.8\textwidth]{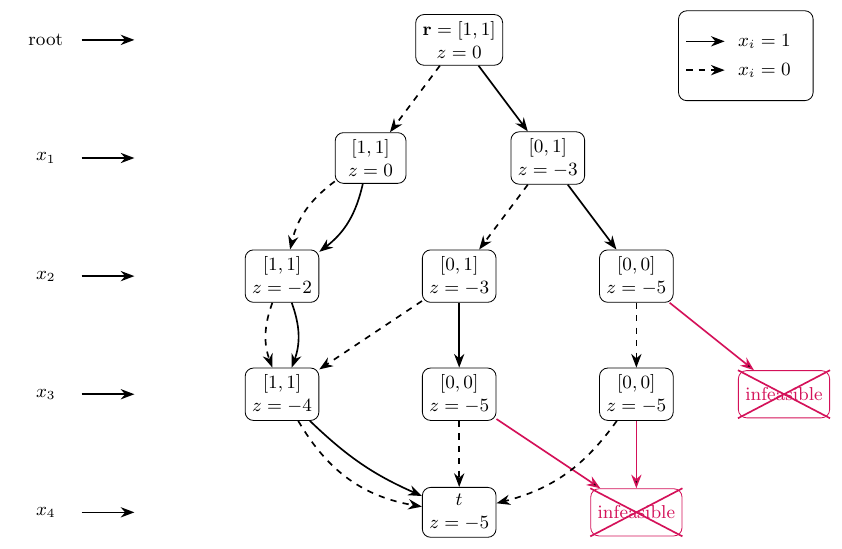}
\caption{Relaxed DD for \Pex, $\width = 3$. Terminal $z$ is a lower bound.}
\label{fig:relaxed_dd}
\end{figure}

\subsubsection{Complexity Analysis}

Algorithm \ref{algo:separation_relaxed} has two phases.

\textbf{Phase 1: Initial $\width$=1 DD:}
For each layer, creates single node $v$ with $\mathcal{O}(K)$ in-arcs (one per parent label) and computes merged state. Per layer:
\[
\textcolor{safeA1}{\mathcal{O}(K \cdot |s|)} + \mathcal{O}(K) \cdot (C_{updt} + C_{mrg})
\]

\textbf{Phase 2: Feasibility Pruning and Node Splitting:}

Arc feasibility pruning checks up to $K \cdot \width$ arcs, calling \textsc{UpdateState} and \textsc{IsFeasible} for each $\big(\textcolor{safeA1}{\mathcal{O}(K \cdot \width \cdot |s|)} + \mathcal{O}(K \cdot \width) \cdot (C_{updt} + C_{isf})\big)$.

Node splitting performs up to $\width - 1$ splits $\big(\mathcal{O}(\width)\big)$.
Each split:
(1) Calls $\textsc{SelectNodeToSplit}$ ($C_{selspl}$).
(2) Calls $\textsc{PartitionArcs}$ ($C_{part}$) to partition in-arcs into sets $I_u$ and $I_{u'}$,
$\big(\mathcal{O}(|I_u| + |I_{u'}|) + C_{part}\big)$.
In the worst case the node with the most in-arcs is selected, and one arc is separated. So the first split processes $\mathcal{O}(K \cdot \width)$ arcs, the next $\mathcal{O}(K \cdot \width) - 1 = \mathcal{O}(K \cdot \width)$ arcs, and the last $\mathcal{O}(K \cdot \width) - (\width - 2) = \mathcal{O}(K \cdot \width)$ arcs. Thus, in the worst case $\mathcal{O}(|I_u| + |I_{u'}|) = \mathcal{O}(K \cdot \width)$.
(3) Redirects in-arcs: Arcs in $I_{u'}$ redirected from $u$ to $u'$ (subsumed by step (2)).
(4) Copies out-arcs: Node $u$ has up to $K$ out-arcs $\big(\mathcal{O}(K)\big)$.
(5) Recomputes states: Computes merged states from the partitioned in-arc sets.
The worst case from step (2) (splitting one arc per iteration), applies here.
Thus this step requires $\mathcal{O}(K \cdot \width \cdot |s|) + \mathcal{O}(K \cdot \width) \cdot (C_{updt} + C_{mrg})$.
The overall phase 2 per layer cost: \[
\textcolor{safeA1}{\mathcal{O}(K \cdot \width^2 \cdot |s|)} + \mathcal{O}(\width) \cdot (C_{selspl} + C_{part}) + \mathcal{O}(K \cdot \width) \cdot C_{isf} + \mathcal{O}(K \cdot \width^2) \cdot (C_{updt} + C_{mrg})
\]

Phase 2 dominates phase 1. The overall complexity for $n$ layers:
\[
n \cdot \Big( \textcolor{safeA1}{\mathcal{O}(K \cdot \width^2 \cdot |s|)} + \mathcal{O}(\width) \cdot (C_{selspl} + C_{part}) + \mathcal{O}(K \cdot \width) \cdot C_{isf} + \mathcal{O}(K \cdot \width^2) \cdot (C_{updt} + C_{mrg}) \Big)
\]

The $\mathcal{O}(\width^2)$ factor arises from performing $\mathcal{O}(\width)$ splits, with each split recomputing a merged state from $\mathcal{O}(K \cdot \width)$ in-arcs.
This quadratic growth is evident in practice.

\begin{algorithm}[h!b]
\SetAlgoLined
    \textbf{Input:} Root node $r$ for a given \Pgen. Problem-specific functions: $\textsc{UpdateState}$, $\textsc{MergeStates}$, $\textsc{IsFeasible}$, $\textsc{SelectNodeToSplit}$, and $\textsc{PartitionArcs}$.\\
    \tcp{Phase 1: Construct initial $\width = 1$ relaxed DD}
    Initialize $L_0 \leftarrow \{r\}$\\
    \For{$i = 1$ to $n$}{
        Let $u$ be the single node in $L_{i-1}$\\
        Create single node $v$ in layer $L_i$\\
        Add arcs from $u$ to $v$ labeled $l$ for each $l \in d(u)$\\    
        Let $l_0$ be the first label in $d(u)$\\
        Initialize $s(v) \leftarrow$ \textsc{UpdateState}$(s(u), x_i, l_0)$\\
        \ForEach{label $l \in d(u) \backslash \{l_0\}$}{
            $s(v) \leftarrow$ \textsc{MergeStates}$(s(v), \textsc{UpdateState}(s(u), x_i, l))$\\
        }
    }

    \tcp{Phase 2: Refine via node splitting}
    \For{$i = 1$ to $n$}{
        \tcp{Feasibility pruning}
        \ForEach{arc $a$ into $L_i$}{
            $s_a \leftarrow$ \textsc{UpdateState}$(s(p(a)), x_i, l(a))$\\
            \lIf{not \textsc{IsFeasible}$(s_a)$}{remove $a$}
        }
        \tcp{Node splitting}
        \While{$|L_i| < \width$}{
            $u \leftarrow$ \textsc{SelectNodeToSplit}$(L_i)$ (break loop if no beneficial split)\\
            $(I_u, I_{u'}) \leftarrow$ \textsc{PartitionArcs}$(u)$ (partition in-arcs into two sets)\\
            Create new node $u'$: $L_i \leftarrow L_i \cup \{u'\}$\\
            Redirect arcs in $I_{u'}$ from $u$ to $u'$\\
            Copy all out-arcs of $u$ to $u'$\\
            \ForEach{$(\beta, I) \in \{(u', I_{u'}), (u, I_u)\}$}{
                Let $a_0$ be the first arc in $I$\\
                $s(\beta) \leftarrow$ \textsc{UpdateState}$(s(p(a_0)), x_i, l(a_0))$\\
                \ForEach{remaining arc $a \in I$}{
                    $s(\beta) \leftarrow$ \textsc{MergeStates}$(s(\beta), \textsc{UpdateState}(s(p(a)), x_i, l(a)))$\\
                }
            }
        }
    }

    \Return{\relDD}
    \caption{Relaxed DD Construction by Separation \cite{DDforO}}
    \label{algo:separation_relaxed}
\end{algorithm}

\subsubsection{Application to Integer Programs}
For \Pint, $|s| = m$.
$C_{updt} = \mathcal{O}(m)$ (computing residuals via Equation \ref{eq:residual_update});
$C_{mrg} = \mathcal{O}(m)$ (element-wise maximum of residuals);
$C_{isf} = \mathcal{O}(m)$ (checking residuals);
$C_{selspl} = \mathcal{O}(\width)$ amortized (sorted nodes: $\mathcal{O}(\width \log \width)$ sort, $\mathcal{O}(\width)$ inserts);
$C_{part} = \mathcal{O}(K \cdot \width)$ (arc partitioning).
So the overall complexity becomes:
\[
n \cdot \Big( \textcolor{safeA1}{\mathcal{O}(K \cdot \width^2 \cdot m)} + \mathcal{O}(K \cdot \width^2) + \mathcal{O}( \width \log \width) + \mathcal{O}(K \cdot \width  \cdot m) + \mathcal{O}(K \cdot \width^2 \cdot m) \Big)
\]
All problem-specific terms are subsumed by inherent costs.

\section{Implicit Decision Diagrams}
\label{sec:implicit_dds}
We present alternative frameworks for constructing restricted and relaxed DDs, improving on the inherent costs from Sections \ref{sec:restricted_dd} and \ref{sec:relaxed_dd}.
Prior work on a DD processing algorithm called Peel and Bound \cite{Rudich2024Thesis}, introduced implicit DDs that store arc labels implicitly; we take this idea to its conclusion by storing the arcs themselves implicitly.
This achieves substantially wider DDs than prior work. 

This section parallels Section~\ref{sec:standard_dds}: we describe and formalize each algorithm, then analyze its complexity.
Since implicit DDs are the paper's main contribution, the exposition is more thorough, and it includes optimality proofs. 

\subsection{Implicit Restricted Decision Diagrams}

Algorithm \ref{algo:topdown_restricted} generates $\mathcal{O}(K \cdot \width)$ children before width limiting, explicitly storing nodes that may be discarded.
We avoid creating such nodes by discriminating based on an objective threshold $\tau$.

\subsubsection{Histogram Construction}
\label{subsubsec:histogram_construction}

Algorithm \ref{algo:histogram} constructs a histogram over arc objective values; it is used for both restricted and relaxed implicit DDs.
Let $f(u,l)$ be the objective value of the arc from $u$ with label $l$, $(u,l)$.
The algorithm finds $(\tau_{min}, \tau_{max})$, the min and max arc values, partitions $[\tau_{min}, \tau_{max}]$ into $\mathcal{O}(\width)$ bins, and assigns each arc to a bin.
The cumulative histogram $C[j,l]$ gives the number of arcs with label $l$ in bins $\leq j$.
Algorithm \ref{algo:histogram} runs in $\mathcal{O}(K \cdot \width)$ time.

\begin{algorithm}[h!t]
\SetAlgoLined
    \textbf{Input:} Parent layer $Q$, width limit $\width$, domain size $K$ \\
    \textbf{Output:} Bin boundaries $(\tau_{min}, \tau_{max})$, cumulative histogram $C$ \\
    \tcp{Pass 1: Compute objective bounds}
    $(\tau_{min}, \tau_{max}) \leftarrow (\min, \max)$ of $f(u,l)$ over all combinations of $u \in Q$, $l 
  \in d(u)$\\
    \tcp{Pass 2: Build histogram}
    $H \leftarrow \width \times K$ matrix of zeros\\
    \ForEach{node $u \in Q$}{
        \ForEach{label $l \in d(u)$}{
            $j \leftarrow \min\Big(\width, \; 1 + \lfloor \frac{f(u,l) - \tau_{min}}{\tau_{max} - \tau_{min}} \cdot \width \rfloor\Big)$ (with $\frac{0}{0} := 0$)\\
            $H[j,l] \leftarrow H[j,l] + 1$\\
        }
    }
    \ForEach{label $l \in \{1, \ldots, K\}$}{
        $C[0,l] \leftarrow 0$\\
        $C[1,l] \leftarrow H[1,l]$\\
        \For{$j = 2$ to $\width$}{
            $C[j,l] \leftarrow C[j-1,l] + H[j,l]$\\
        }
    }
    \Return{$(\tau_{min}, \tau_{max}, C)$}
    \caption{Histogram Construction for Arc Objective Values}
    \label{algo:histogram}
\end{algorithm}

\subsubsection{Two-Zone Selective Child Generation}

Using Algorithm \ref{algo:histogram}, we select threshold $\tau$ as the largest bin where $\sum_l C[\tau, l] \leq \width$.
This approximates the $\width$-th smallest arc objective value in $\mathcal{O}(K \cdot \width)$ time, avoiding $\mathcal{O}(K \cdot \width \cdot \log(K \cdot \width))$ sorting.

Algorithm \ref{algo:implicit_restricted} formalizes a \textit{two-zone} approach, processing both zones in a single pass:
(1) the \textit{baseline zone} includes all arcs in bins $\leq \tau$, containing at most $\width$ arcs;
(2) the \textit{extras zone} uses a budget to greedily select arcs from bin $\tau + 1$.
This eliminates merging and width limiting from Algorithm \ref{algo:topdown_restricted}.
When the number of feasible arcs exceeds $\width$, this ensures $|Q_{next}| = \width$; otherwise all arcs are kept.
The only tradeoff vs. standard width-limiting is that arcs from bin $\tau + 1$ are selected arbitrarily.
Figure~\ref{fig:implicit_restricted} demonstrates the procedure for \Pex.

\begin{algorithm}[h!t]
\SetAlgoLined
    \textbf{Input:} The root node $r$ for a given \Pgen, width limit $\width$, domain size $K$. Problem-specific functions:  $\textsc{UpdateState}$ and $\textsc{IsFeasible}$.\\
    Initialize $Q \leftarrow \{r\}$, $Q_{next} \leftarrow \emptyset$\\

    \While{$Q \neq \emptyset$}{
        $Q_{next} \leftarrow \emptyset$\\
        \tcp{1. Histogram construction and threshold selection}
        $(\tau_{min}, \tau_{max}, C) \leftarrow$ Algorithm \ref{algo:histogram}$(Q, \width, K)$\\
        $\tau \leftarrow \max\{j : \sum_l C[j,l] \leq \width\}$\\
        \tcp{2. Two-zone selective child generation}
        $\text{budget} \leftarrow \width - \sum_l C[\tau, l]$\\
        \ForEach{node $u \in Q$}{
            \ForEach{label $l \in d(u)$}{
                $j \leftarrow \min\Big(\width, \; 1 + \lfloor \frac{f(u,l) - \tau_{min}}{\tau_{max} - \tau_{min}} \cdot \width \rfloor\Big)$\\
                \If{$j \leq \tau$ \textbf{or} $(j = \tau + 1$ \textbf{and} $\text{budget} > 0)$}{
                    Create node $v$: $Q_{next} \leftarrow Q_{next} \cup \{v\}$\\
                    $s(v) \leftarrow \textup{\textsc{UpdateState}}(s(u), x_{\ell(u)+1}, l)$\\
                    \lIf{$j = \tau + 1$}{$\text{budget} \leftarrow \text{budget} - 1$}
                }
            }
        }
        \tcp{3. Feasibility pruning}
        \ForEach{node $v \in Q_{next}$}{
            \If{\textup{\textsc{IsFeasible}}($s(v)$) $=$ false}{
                $Q_{next} \leftarrow Q_{next} \backslash \{v\}$\\
            }
        }
        \textbf{(Optional)} Apply additional operations (e.g., rough bounding \cite{ddo})\\

        $Q \leftarrow Q_{next}$\\
    }

    \Return{\resDD}
    \caption{Implicit Restricted Decision Diagrams}
    \label{algo:implicit_restricted}
\end{algorithm}

\begin{figure}[h!t]
\centering
\includegraphics[width=0.75\textwidth]{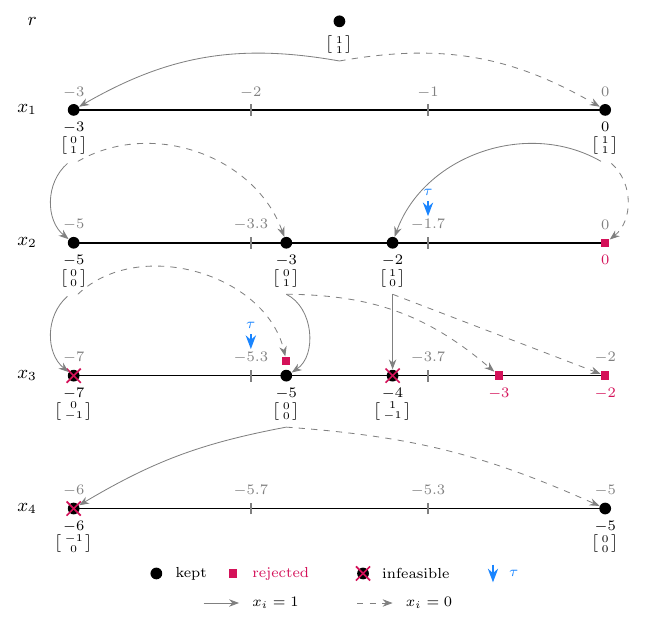}
\caption{Implicit restricted DD for \Pex with $\width=3$. Arcs below threshold $\tau$ are kept (dots); excess arcs are rejected (squares). 
The gray ticks on the number line indicate histogram bin edges.
On $x_2$, one of four candidates is rejected. On $x_3$, one node is rejected arbitrarily from the bin above $\tau$, and two of the kept nodes are infeasible (X). 
Though visually leftmost (lowest objective), the rejected node is rightmost in memory order (highest parent index), so the budget expires before it is processed.
Only one feasible node remains for $x_4$.
}
\label{fig:implicit_restricted}
\end{figure}

\subsubsection{Generalized Complexity Analysis for Implicit Restricted DDs}
We extract the inherent costs to compare with Algorithm \ref{algo:topdown_restricted}.
The algorithm performs three operations per layer: (1) histogram construction and threshold selection, (2) selective child generation, and (3) feasibility pruning.

\textbf{Histogram construction and threshold selection:}
Algorithm \ref{algo:histogram} runs in $\mathcal{O}(K \cdot \width)$ time.
In Algorithm \ref{algo:implicit_restricted}, $\tau$ is found by scanning the $\width$ cumulative totals $\big(\mathcal{O}(\width)\big)$.
Per layer: \[
    \textcolor{safeA1}{\mathcal{O}(K \cdot \width)}
\]

\textbf{Selective child generation:}
Evaluates all $\mathcal{O}(K \cdot \width)$ potential arcs in a single pass, creating at most $\width$ children.
For each selected arc, $\textsc{UpdateState}$ creates child $v$ $\big(\mathcal{O}(|s|) + C_{updt}\big)$.
Per layer: \[
    \textcolor{safeA1}{\mathcal{O}(K \cdot \width + \width \cdot |s|)} + \mathcal{O}(\width) \cdot C_{updt}
\]

\textbf{Feasibility pruning:}
Checks all $\mathcal{O}(\width)$ child nodes for feasibility.
For each child, $\textsc{IsFeasible}$ checks the state $\big(\mathcal{O}(|s|) + C_{isf}\big)$.
Per layer: \[
    \textcolor{safeA1}{\mathcal{O}(\width \cdot |s|)} + \mathcal{O}(\width) \cdot C_{isf}
\]

\textbf{Additional operations:} Optional operations with cost $C_{othr}$ (excluded).
\vspace{1em}\\
The inherent framework cost is the \textcolor{safeA1}{component} that excludes problem-specific costs ($C_{updt}$, $C_{isf}$, $C_{othr}$).
Comparing implicit DDs (Algorithm \ref{algo:implicit_restricted}) to standard DDs (Algorithm \ref{algo:topdown_restricted}) for $n$ layers:
\begin{align*}
\textbf{Standard:} \quad & \textcolor{safeA1}{\mathcal{O}\Big(n \cdot K \cdot \width \cdot \max\big(|s|, \log(K \cdot \width)\big)\Big)} + \mathcal{O}(n \cdot K \cdot \width) \cdot (C_{updt} + C_{isf}) \\
\textbf{Implicit:} \quad & \textcolor{safeA1}{\mathcal{O}\Big(n \cdot \width \cdot (K + |s|)\Big)} + \mathcal{O}(n \cdot \width) \cdot (C_{updt} + C_{isf})
\end{align*}
Factoring out $\mathcal{O}(n \cdot \width)$ and subtracting  $(C_{updt} + C_{isf})$:
\[
\textcolor{safeA1}{\mathcal{O}\Big(K \cdot \max\big(|s|, \log(K \cdot \width)\big)\Big)} + \mathcal{O}(K) \cdot (C_{updt} + C_{isf}) \gg \textcolor{safeA1}{\mathcal{O}\Big(K + |s|\Big)}
\]

When $\log(K \cdot \width) > |s|$, standard incurs $\mathcal{O}(K \cdot \log(K \cdot \width))$ from width limiting, while implicit achieves $\mathcal{O}(K + |s|)$, eliminating the log factor.
When $|s| \geq \log(K \cdot \width)$, the comparison becomes $\mathcal{O}(K \cdot |s|)$ versus $\mathcal{O}(K + |s|)$, reducing $|s|$ from a multiplicative factor to an additive factor.
Algorithm \ref{algo:implicit_restricted} scales linearly in $\width$; Algorithm \ref{algo:topdown_restricted} scales as $\mathcal{O}(\width \cdot \log(\width))$.

\subsubsection{Application to Integer Programs}
For \Pint, state and operations match Section \ref{sec:restricted_dd}, except \textsc{UpdateState} is called $\leq \width$ times per layer, reducing reads from $\mathcal{O}(K \cdot \width \cdot m)$ to $\mathcal{O}(\width \cdot m)$.
As $|s| = m$, the complexity is:
\[
\mathcal{O}\Big(n \cdot \width \cdot (K + m)\Big)
\]

\subsection{Implicit Relaxed Decision Diagrams}
\label{subsec:implicit_rel_dds}

Relaxed DD construction by separation (Section \ref{sec:relaxed_dd}) has $\mathcal{O}(K \cdot \width^2 \cdot |s|)$ inherent cost per layer because intermediate states are generated explicitly.
We present an alternative framework with $\mathcal{O}(K \cdot \width \cdot |s|)$ inherent cost per layer by using \textit{interval-based arc representation} and \textit{threshold-based node splitting}.
We first explain the algorithm in detail, then formalize it in Algorithm \ref{algo:implicit_relaxed}.

\subsubsection{Interval-Based Arc Storage}
Standard construction stores arcs explicitly; we store them as intervals.
We construct our initial diagram with up to $K$ nodes instead of $1$, such that all in-arcs to a node have the same label, and the nodes are sorted by their associated label.
First, feasible domain bounds: we define $\mathcal{D}[p] = [\mathit{lb}_p, \mathit{ub}_p]$ as the constraint-feasible label range for children of $p$, where $\mathit{lb}_p$ is the smallest feasible label, and $\mathit{ub}_p$ is the largest.
We define a subroutine \textsc{ComputeDomainInterval} to compute these via problem-specific operations (\textsc{UpdateState}, \textsc{IsFeasible}).
Let $\text{idx}(u)$ be the index of node $u$ in its layer, and $l(u)$ be the label of all the arcs that point to $u$.
Second, inverted intervals: each child node with label $l$ stores $[\psi(l), \phi(l)]$, the interval of parent indices where $l$ is feasible.
Let the \textit{inverted interval} for $l$ be:
\begin{align*}
\psi(l) &= \min \Big\{\text{idx}(p) : \;\; p \in L_{i-1}, \;\; l \in \mathcal{D}[p]\Big\} \\
\phi(l) &= \max \Big\{\text{idx}(p) : \;\; p \in L_{i-1}, \;\; l \in \mathcal{D}[p]\Big\}
\end{align*}
This interval may include infeasible parents (a valid relaxation), but no information is lost; individual arcs can be verified via \textsc{IsArcFeasible}, defined via \textsc{UpdateState} and \textsc{IsFeasible}.
In the final DD, each node $u$ with label $l$ stores in-arcs as $\psi(l), \phi(l)$ and out-arcs as $\mathit{lb}_u, \mathit{ub}_u$, avoiding the $\mathcal{O}(\width)$ in-arc pointers and $\mathcal{O}(K)$ out-arc pointers used by standard DDs.
\begin{theorem}[IP Exactness]\label{thm:ip_exactness}
For IPs, the interval representation is exact:
\[
\text{idx}(p) \in [\psi(l), \phi(l)] \quad \text{AND} \quad l \in \mathcal{D}[p] \iff \text{arc } (p,l) \text{ is feasible}
\]
\end{theorem}
\noindent The proof appears after Lemma~\ref{lem:ip_domain_exactness}.

\subsubsection{Width-$K$ Construction}
We require $\width \geq K$ since we start at width-$K$.
We iteratively construct the width-$K$ DD starting from the root node $r$. Given parent layer $L_{i-1}$ with $\leq K$ nodes, layer $L_i$ is constructed as follows:
\textbf{(1) Compute feasible domains:}
For each parent $p \in L_{i-1}$, compute $\mathcal{D}[p] = [\mathit{lb}_p, \mathit{ub}_p]$ via \textsc{ComputeDomainInterval}.
\textbf{(2) Invert intervals:}
For each label $l \in d(x_i)$, compute $[\psi(l), \phi(l)]$.
\textbf{(3) Create child nodes:}
For each label $l \in d(x_i)$ where $\psi(l)$ and $\phi(l)$ are defined, create a child node with label $l$ and inverted interval $[\psi(l), \phi(l)]$.
\textbf{(4) Update node states:}
For each node, compute its state via \textsc{ComputeState}, defined via \textsc{UpdateState}, \textsc{IsFeasible}, and \textsc{MergeStates}.
Algorithm \ref{algo:width_k_construction} formalizes this procedure, and Figure~\ref{fig:implicit_relaxed_dd} (left) shows it for \Pex.

\begin{figure}[h]
\centering
\includegraphics[width=\textwidth]{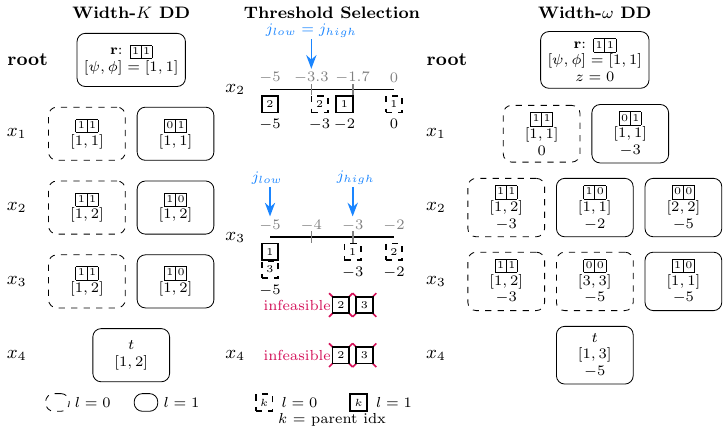}
\caption{Implicit relaxed DD construction for $\Pex$ with $\width=3$.
The width-$K$ DD (left) is constructed first.
Then, for each layer, we select thresholds (middle) and generate width-$\width$ children (right).
\textbf{Left:} Width-$K$ DD; each node shows state and inverted interval.
\textbf{Middle:} Arcs placed on a number line by objective value; thresholds $j_{low}$ and $j_{high}$ determine which arcs are individualized versus merged.
For $x_2$, there are four arcs at $z \in \{-5, -3, -2, 0\}$ with threshold $j_{low}=j_{high}=-3.3$: the $z=-5$ arc will be individualized, and the remaining arcs will be merged by label.
\textbf{Right:} Width-$\width$ DD; each node shows state, inverted interval, and objective value.
For $x_2$, the $l=0$ node with interval $[1,2]$ is the merged node (the interval indicates it came from parents 1 and 2).}
\label{fig:implicit_relaxed_dd}
\end{figure}

\begin{algorithm}[h!t]
\SetAlgoLined
    \textbf{Input:} The root node $r$ for a given \Pgen with initial state $s(r)$.\\
    $L_0 \leftarrow \{r\}$ \\
    $\relDD \leftarrow \{L_0\}$ \\
    \For{$i = 1$ to $n$}{
        \ForEach{parent node $p \in L_{i-1}$}{
            $\mathcal{D}[p] = [\mathit{lb}_p, \mathit{ub}_p] \leftarrow$ \textsc{ComputeDomainInterval}$(p)$\\
        }
        $L_i \leftarrow \emptyset$\\
        \ForEach{label $l \in d(x_i)$}{
            $\psi(l) \leftarrow \min \{\text{idx}(p) : l \in \mathcal{D}[p]\}$, $\;$ $\phi(l) \leftarrow \max \{\text{idx}(p) : l \in \mathcal{D}[p]\}$\\
            Create node $q$ with label $l$ and inverted interval $[\psi(l), \phi(l)]$: $L_i \leftarrow L_i \cup \{q\}$\\
        }
        \ForEach{node $q \in L_i$ with label $l$ and inverted interval $[\psi(l), \phi(l)]$}{
            $s(q) \leftarrow$ \textsc{ComputeState}$(L_{i-1}, [\psi(l), \phi(l)], l)$\\
        }
        $\relDD \leftarrow \relDD \cup\{L_i\}$
    }

    \Return{$\relDD$}
    \caption{Initial Width-$K$ Relaxed DD Construction}
    \label{algo:width_k_construction}
\end{algorithm}

\subsubsection{\textsc{IsArcFeasible}: General and IP}

Given parent $p$ and label $l$, determine if arc $(p,l)$ is feasible.
Equivalent to \textsc{IsFeasible}$(\textsc{UpdateState}(s(p), x_i, l))$.
The complexity is bounded by $\mathcal{O}(|s|)$ for state access, plus $C_{updt}$ and $C_{isf}$.
For IPs, once the intervals are computed, \textsc{IsArcFeasible} is $\mathcal{O}(1)$ by Theorem~\ref{thm:ip_exactness}.

\subsubsection{\textsc{ComputeDomainInterval}: General and IP}

Given parent $p$ on $L_{i-1}$, compute the interval $\mathcal{D}[p] = [\mathit{lb}_p, \mathit{ub}_p]$ of feasible values for $x_i$ from $p$.
This finds:
\[\mathit{lb}_p = \min\{l : \textsc{IsArcFeasible}(p, l)\} \text{ and } \mathit{ub}_p = \max\{l : \textsc{IsArcFeasible}(p, l)\}\]
In general, this requires iterating over the domain $\big(\mathcal{O}(K)\big)$, calling \textsc{IsArcFeasible} for each value.
Thus, the complexity is bounded by $\mathcal{O}(K \cdot |s|) + \mathcal{O}(K) \cdot (C_{updt} + C_{isf})$.
For IPs, each constraint derives a bound on $x_i$ from its residual; taking the tightest requires $\mathcal{O}(|s|) = \mathcal{O}(m)$.
\begin{lemma}\label{lem:ip_domain_exactness}
For IPs, $\mathcal{D}[p] = \{l \in \mathbb{Z} : \textsc{IsArcFeasible}(p, l)\}$.
\end{lemma}
\begin{proof}
Linear constraints define a convex region, so if $\mathit{lb}_p$ and $\mathit{ub}_p$ are feasible (non-negative residuals), then every integer $l \in [\mathit{lb}_p, \mathit{ub}_p]$ is feasible. \qed
\end{proof} 

\begin{proof}[Theorem~\ref{thm:ip_exactness}]
By Lemma~\ref{lem:ip_domain_exactness}, $l \in \mathcal{D}[p]$ determines if arc $(p,l)$ exists;\\ $\text{idx}(p) \in [\psi(l), \phi(l)]$ determines which node it connects to.
Since intervals are non-overlapping, exactly one node with label $l$ contains each parent, so the checks together identify the unique arc. \qed
\end{proof}

\subsubsection{\textsc{ComputeState}: General and IP}

Given inverted interval $[\psi(l), \phi(l)]$, label $l$, and parent layer $L_{i-1}$, compute $s(q)$ by merging all feasible parent states.
For each $p$ with $\text{idx}(p) \in [\psi(l), \phi(l)]$, compute \textsc{UpdateState}$(s(p), x_i, l)$, check \textsc{IsFeasible}, and if feasible, merge via \textsc{MergeStates}.
Let $|I|$ denote the number of parents in the interval.
The complexity is $\mathcal{O}(|I| \cdot |s|) + \mathcal{O}(|I|) \cdot (C_{updt} + C_{isf} + C_{mrg})$.
For IPs, \textsc{IsArcFeasible} is $\mathcal{O}(1)$ by Theorem~\ref{thm:ip_exactness}, so we check feasibility first, calling \textsc{UpdateState} and \textsc{MergeStates} only for feasible arcs.
IP states are residuals updated via Equation~\ref{eq:residual_update}, and merged via element-wise maximum.

\subsubsection{Efficient Node Splitting}

The initial width-$K$ DD has one node per domain value, each storing inverted interval $[\psi(l), \phi(l)]$ spanning up to $K$ parents.
After splitting to width $\width$ via threshold-based splitting (below), parent layers may contain up to $\width$ nodes, so each child's inverted interval may span up to $\width$ parents.
However, our procedure ensures that when a node from the $K$-width layer splits, the resulting nodes partition the parent layer into \textbf{non-overlapping intervals}.
These intervals span at most $\width$ parents in total.
Thus, \textsc{ComputeState} requires $\mathcal{O}(K \cdot \width)$ parent accesses per layer, avoiding the $\mathcal{O}(\width^2)$ of standard relaxed DDs.

\subsubsection{Threshold Selection}

Given a threshold $\tau$, each arc with $f(p,l) < \tau$ becomes a node, while each contiguous sequence of arcs with $f(p,l) \geq \tau$ (consecutive parents, same label) shares a single node.
Using Algorithm~\ref{algo:histogram}, we bin arcs into $\mathcal{O}(\width)$ bins, then compute bounds on the layer size for each candidate threshold to select a $\tau$ yielding $\leq \width$ nodes.

For bin boundary $j$, consider one of the $K$ nodes from Algorithm~\ref{algo:width_k_construction}.
Let $n_{low}$ be the number of in-arcs to this node in bins $< j$ (arcs that are always individualized), and $n_{high}$ the number in bins $\geq j$ (arcs that may be grouped).
If $n_{high} = 0$, the node splits into exactly $n_{low}$ nodes (all arcs individualized); if $n_{low} = 0$, it doesn't split.
When both $n_{low} > 0$ and $n_{high} > 0$, the node count is bounded by $[n_{low} + 1, n_{low} + \min(n_{low} + 1, n_{high})]$, derived as follows. 
The lower bound assumes above-threshold arcs form one node; the upper bound treats individualized arcs as separators creating at most $n_{low} + 1$ contiguous runs of above-threshold arcs.
Summing over all $K$ nodes yields global layer size bounds for threshold $j$.
Computing these bounds is $\mathcal{O}(K)$ for each of the $\mathcal{O}(\width)$ candidate thresholds, totaling $\mathcal{O}(K \cdot \width)$ operations.
We select $j_{low}$ as the highest bin with upper bound $\leq \width$, and $j_{high}$ as the highest bin with lower bound $\leq \width$.
Arcs with $j < j_{low}$ are individualized; contiguous runs of arcs with $j \geq j_{low}$ share a single node.
Arcs with $j_{low} \leq j < j_{high}$ are individualized until none remain or $|L| = \width$.
Figure~\ref{fig:implicit_relaxed_dd} (middle) demonstrates this for \Pex.

\subsubsection{Layer Construction}
The two passes described above are formalized in Algorithm~\ref{algo:implicit_relaxed}; Figure~\ref{fig:implicit_relaxed_dd} (right) shows the result for \Pex.
Our implementation combines them into a single pass using a budget to track remaining capacity.
For each node, states are computed using \textsc{ComputeState}, then domain intervals are computed using \textsc{ComputeDomainInterval}.
The inverted intervals $[\psi(l), \phi(l)]$ are recomputed from these domain intervals for use in the next layer.
Since each of the $\mathcal{O}(\width)$ nodes merges $\mathcal{O}(K)$ parents on average, the cost of calling \textsc{ComputeState} $\width$ times is $\mathcal{O}(K \cdot \width \cdot |s|) + \mathcal{O}(K \cdot \width) \cdot (C_{updt} + C_{isf} + C_{mrg})$.

\begin{theorem}[Tightness]\label{thm:tightness}
The complexity is tight: $\Theta(K \cdot \width \cdot |s|)$ when \textsc{UpdateState} and \textsc{MergeStates} are black-box operations over $|s|$-component states.
\end{theorem}
\begin{proof}
When \textsc{UpdateState} is not a homomorphism with respect to \textsc{MergeStates}, parent states cannot be merged before updating.
The algorithm must invoke \textsc{UpdateState} on each of the $\mathcal{O}(K \cdot \width)$ arcs.
Each invocation reads $\Omega(|s|)$ components since the output may depend on any input.
For example, $\textsc{UpdateState}(s, v) = s \cdot v$ with $\textsc{MergeStates}(s_1,s_2) = \max(s_1, s_2)$ yields different results depending on the order they are applied when $v < 0$. \qed
\end{proof}

\begin{algorithm}[h!b]
\SetAlgoLined
    $\relDD \leftarrow$ width-$K$ relaxed DD (Algorithm \ref{algo:width_k_construction})\\

    \For{$i = 1$ to $n$}{
        \tcp{Threshold selection}
        $(\tau_{min}, \tau_{max}, C) \leftarrow$ Algorithm \ref{algo:histogram}$(L_{i-1}, \width, K)$\\
        \ForEach{bin $j$, node $q \in L_i$}{
            $n_{low} \leftarrow C[j-1,l(q)]$; $n_{high} \leftarrow C[\width,l(q)] - n_{low}$
        }
        $l[j] \leftarrow \sum_q (n_{low} + \min(1, n_{high}))$; \quad $u[j] \leftarrow \sum_q (n_{low} + \min(n_{low}+1, n_{high}))$\\
        $j_{low} \leftarrow \max\{j : u[j] \leq \width\}$; \quad $j_{high} \leftarrow \max\{j : l[j] \leq \width\}$\\

        \tcp{Pass 1: individualize safe arcs}
        $L_i \leftarrow \emptyset$\\
        \ForEach{label $l \in d(x_i)$ in sorted order}{
            $\psi(l) \leftarrow \min \{\text{idx}(q) : l \in \mathcal{D}[q]\}$; $\phi(l) \leftarrow \max \{\text{idx}(q) : l \in \mathcal{D}[q]\}$\\
            $\text{start} \leftarrow \emptyset$\\
            \ForEach{node $p \in L_{i-1}$ with $\text{idx}(p) \in [\psi(l), \phi(l)]$}{
                $j \leftarrow$ bin of $f(p,l)$\\
                \uIf{$j < j_{low}$}{
                    \If{$\text{start} \neq \emptyset$}{
                        \tcp{a node is defined by its parents and label}
                        add node$([\text{start}, idx(p)-1], l)$; $\text{start} \leftarrow \emptyset$
                    }
                    add node$(idx(p),l)$
                }
                \lElseIf{$\text{start} = \emptyset$}{$\text{start} \leftarrow idx(p)$}
            }
            \lIf{$\text{start} \neq \emptyset$}{add node$([\text{start}, \phi(l)], l)$}
        }

        \tcp{Pass 2: fill remaining width}
        \ForEach{node $q \in L_i$}{
            \ForEach{arc $(p, l(q))$ with $j_{low} \leq j < j_{high}$}{
                \lIf{$|L_i| \geq \width -1$}{\textbf{stop}}
                split $q$ to individualize arc $(p, l(q))$
            }
        }

        \ForEach{node $q \in L_i$}{
            $s(q) \leftarrow$ \textsc{ComputeState}$(L_{i-1}, \text{arcs}(q), l(q))$\\
            $\mathcal{D}[q] = [\mathit{lb}_q, \mathit{ub}_q] \leftarrow$ \textsc{ComputeDomainInterval}$(q)$
        }
        \textbf{(Optional)} Apply additional operations (e.g., rough bounding \cite{ddo})\\
        $\relDD \leftarrow \relDD \cup \{L_i\}$
    }

    \Return{$\relDD$}
    \caption{Implicit Relaxed Decision Diagram Refinement}
    \label{algo:implicit_relaxed}
\end{algorithm}

\subsubsection{Generalized Complexity Analysis for Implicit Relaxed DDs}
We analyze Algorithm \ref{algo:implicit_relaxed}.
The cost of Algorithm \ref{algo:width_k_construction} is subsumed by Algorithm \ref{algo:implicit_relaxed}, as Phase 1 is subsumed by Phase 2 in Algorithm \ref{algo:separation_relaxed}.
Algorithm  \ref{algo:implicit_relaxed} performs three operations per layer: (1) threshold selection, (2) threshold-based splitting, and (3) other.

\textbf{Threshold selection:}
Algorithm~\ref{algo:histogram} runs in $\mathcal{O}(K \cdot \width)$ time.
Computing bounds $l[j]$ and $u[j]$ for each bin boundary requires $\mathcal{O}(K \cdot \width)$ operations.
Finding $j_{low}$ and $j_{high}$ via linear scan is $\mathcal{O}(\width)$.
Per layer: \[
    \textcolor{safeA1}{\mathcal{O}(K \cdot \width)}
\]

\textbf{Layer construction via threshold-based splitting:}
Node splitting requires two passes over $\mathcal{O}(K \cdot \width)$ arcs with $\mathcal{O}(1)$ work per arc, creating at most $\width$ nodes.
By Theorem~\ref{thm:tightness}, calling \textsc{ComputeState} $\width$ times costs $\mathcal{O}(K \cdot \width \cdot |s|) + \mathcal{O}(K \cdot \width) \cdot (C_{updt} + C_{isf} + C_{mrg})$.
Calling \textsc{ComputeDomainInterval} $\width$ times costs $\mathcal{O}(K \cdot \width \cdot |s|) + \mathcal{O}(K \cdot \width) \cdot (C_{updt} + C_{isf})$.
Inverting domain intervals to compute $[\psi(l), \phi(l)]$ for each label is $\mathcal{O}(K \cdot \width)$.
Per layer: \[
    \textcolor{safeA1}{\mathcal{O}(K \cdot \width \cdot |s|)} + \mathcal{O}(K \cdot \width) \cdot (C_{updt} + C_{isf} + C_{mrg})
\]

\textbf{Additional operations:} Optional operations with cost $C_{othr}$ (excluded).

The inherent framework cost is the \textcolor{safeA1}{component} that excludes problem-specific costs.
For $n$ layers:
\begin{align*}
\textbf{Separation:} \quad & \textcolor{safeA1}{\mathcal{O}\Big(n \cdot K \cdot \width^2 \cdot |s|\Big)} + \mathcal{O}(n \cdot \width) \cdot (C_{selspl} + C_{part}) \\
& + \mathcal{O}(n \cdot K \cdot \width) \cdot C_{isf} + \mathcal{O}(n \cdot K \cdot \width^2) \cdot (C_{updt} + C_{mrg}) \\[6pt]
\textbf{Top-down:} \quad & \textcolor{safeA1}{\mathcal{O}\Big(n \cdot K \cdot \width \cdot (K \cdot \width + |s|)\Big)} \\
& + \mathcal{O}(n \cdot K \cdot \width) \cdot (C_{updt} + C_{isf} + C_{mrg} + C_{selmrg}) \\[6pt]
\textbf{Implicit:} \quad & \textcolor{safeA1}{\mathcal{O}\Big(n \cdot K \cdot \width \cdot |s|\Big)} + \mathcal{O}(n \cdot K \cdot \width) \cdot (C_{updt} + C_{isf} + C_{mrg})
\end{align*}

Implicit relaxed DDs (Algorithm \ref{algo:implicit_relaxed}) outperform relaxed DDs by separation (Algorithm \ref{algo:separation_relaxed}).
Removing common factors and simplifying:
\[
\textcolor{safeA1}{\mathcal{O}\Big(\width\Big)} + C_{selspl} + C_{part} + \mathcal{O}(\width) \cdot (C_{updt} + C_{mrg}) \gg \textcolor{safeA1}{\mathcal{O}\Big(1\Big)}
\]

Similarly, Algorithm \ref{algo:implicit_relaxed} outperforms top-down relaxed DDs (Appendix \ref{appendix:topdown_relaxed}, Algorithm \ref{algo:topdown_relaxed}).
Removing common factors and simplifying:
\[
\textcolor{safeA1}{\mathcal{O}\Big(K \cdot \width\Big)} + C_{selmrg} \gg \textcolor{safeA1}{\mathcal{O}\Big(1\Big)}
\]

The key improvement is the elimination of the $\width^2$ factor.
Separation incurs $\mathcal{O}(\width^2)$ pointer updates during node splitting; top-down incurs $\mathcal{O}(K^2 \cdot \width^2)$ arc redirections during merging.
Implicit relaxed DDs avoid these quadratic costs by representing arcs as intervals rather than explicit pointers.

\begin{corollary}[Optimality]
Under the assumptions of Theorem~\ref{thm:tightness}, the inherent cost of Algorithm~\ref{algo:implicit_relaxed} is asymptotically optimal.
\end{corollary}
\begin{proof}
The framework achieves $\mathcal{O}(K \cdot \width \cdot |s|)$ per layer.
By Theorem~\ref{thm:tightness}, this is unavoidable when update and merge are black-box operations. \qed
\end{proof}

\begin{theorem}[Feasibility Lower Bound]\label{thm:feasibility}
$\mathcal{O}(\width) \cdot C_{isf}$ feasibility checks per layer are unavoidable when \textsc{IsFeasible} is a black-box predicate and the algorithm must distinguish feasible from infeasible nodes.
\end{theorem}
\begin{proof}
Without invoking \textsc{IsFeasible}, the algorithm cannot determine whether a node's state is feasible.
A relaxed DD that cannot distinguish feasible from infeasible nodes can compute bounds, but cannot identify feasible solutions or generate meaningful subproblems for branch-and-bound.
To distinguish feasibility for $\width$ output nodes, at least $\width$ invocations are required. \qed
\end{proof}

\begin{corollary}[Problem-Specific Gap]
The gap between Algorithm~\ref{algo:implicit_relaxed} and a theoretically perfect algorithm is a factor $\mathcal{O}(K)$  $\textsc{IsFeasible}$ calls per layer.
\end{corollary}
\begin{proof}
By Theorem~\ref{thm:tightness}, $\mathcal{O}(K \cdot \width) \cdot (C_{updt} + C_{mrg})$ is unavoidable.
By Theorem~\ref{thm:feasibility}, $\mathcal{O}(\width) \cdot C_{isf}$ is unavoidable.
The framework achieves $\mathcal{O}(K \cdot \width) \cdot (C_{updt} + C_{isf} + C_{mrg})$.
Thus, the gap is a factor $\mathcal{O}(K)$. \qed
\end{proof}
This gap reflects checking feasibility at arcs rather than nodes; we hypothesize that the resulting improvement in bound quality makes this tradeoff worthwhile.

\subsubsection{Application to Integer Programs}
For \Pint, the state representation and operations remain identical to Section \ref{sec:relaxed_dd}.
As before $|s| = m$, $C_{updt} = C_{isf} = C_{mrg} = \mathcal{O}(m)$, and all problem-specific operations are subsumed by inherent cost.
The overall compelxity is: 
\[
\mathcal{O}\Big(n \cdot K \cdot \width \cdot m\Big)
\]

\section{Mixed Integer Programming}
\label{sec:mip}

DDs can drive a MIP solver by providing combinatorial relaxations over integers while treating continuous variables separately.

\subsection{MIP Formulation}

Consider a MIP with $n_I$ integer variables and $n_C$ continuous variables.
Partition the constraint matrix $A = [A_I \mid A_C]$ and objective $c = [c_I; c_C]$ accordingly.
Let $\mathit{lb}_C, \mathit{ub}_C$ be the continuous variable bounds.
The initial residuals (Equation~\ref{eq:initial_residuals}) now include continuous contributions.
Define the \emph{continuous contribution} $\delta_i^C$ as:
\begin{equation}
\label{eq:cont_contribution}
\delta_i^C = \sum_{j=1}^{n_C} a_{ij}^C \cdot
\begin{cases}
\mathit{lb}_j^C & \text{if } a_{ij}^C > 0\\
\mathit{ub}_j^C & \text{if } a_{ij}^C < 0
\end{cases}
\end{equation}
Then $r_i^{(0)} = b_i - \delta_i^I - \delta_i^C$, where $\delta_i^I$ is the integer contribution (Equation~\ref{eq:initial_residuals}).

\subsection{DD-Driven MIP Solving}

Build the DD over integers only, propagating residuals via Equation~\ref{eq:residual_update}.
Terminal residuals encode slack after integers are fixed, but with $\boldsymbol{\delta}^C$ still subtracted.

\paragraph{Restricted DDs.}
Normally the best solution that reaches the terminal layer of a restricted DD is the best solution found by that restricted DD.
However for MIP, the continuous variables have to be accounted for.
Instead of merging the nodes on the terminal layer into one node, they have to be hadled separately.
Each terminal node represents a complete integer assignment with residual vector $\mathbf{r}$.
To obtain the true RHS for the continuous LP, add back $\boldsymbol{\delta}^C$:
\begin{equation}
\label{eq:restricted_lp}
\min \; c_C^\top x_C \quad \text{s.t.} \quad A_C x_C \leq \mathbf{r} + \boldsymbol{\delta}^C, \quad \mathit{lb}_C \leq x_C \leq \mathit{ub}_C
\end{equation}
If the LP for a particular terminal node is feasible, the integer assignment plus LP solution yields a feasible MIP solution; otherwise, the integer assignment is infeasible.
A restricted DD can only be marked exact if every terminal node's LP is checked. 

\paragraph{Relaxed DDs.}
Merge all terminal nodes into one, combining residuals via element-wise maximum.
Let $\mathbf{r}^{(1)}, \ldots, \mathbf{r}^{(k)}$ be the residuals before merging, and let $r^{mrg}_i = \max_j r^{(j)}_i$.
Then $\mathbf{r}^{mrg} \geq \mathbf{r}^{(j)}$ component-wise for all $j$.
Solve the LP from Equation \eqref{eq:restricted_lp} with $\mathbf{r} = \mathbf{r}^{mrg}$ to obtain $z^*_{LP}$.
Since larger RHS weakens the constraints, $z^*_{LP}(\mathbf{r}^{mrg}) \leq z^*_{LP}(\mathbf{r}^{(j)})$ for all $j$.
Thus the shortest root-to-terminal path in the relaxed DD, plus $z^*_{LP}$, is a valid lower (relaxed) bound on the MIP.

\paragraph{Convergence.}
Restricted and relaxed DDs yield upper and lower bounds, respectively.
A node is \emph{exact} if all paths to it yield the same state; an \emph{exact cutset} is a set of exact nodes intersecting every root-to-terminal path.
DD-based branch-and-bound \cite{BnB,BnB2} prunes nodes whose lower bound exceeds the best upper bound, and branches on exact cutsets.
For MIP, bounds are two-part (integer contribution plus LP contribution), but convergence follows from the same arguments.

\section{Experimental Results}
\label{sec:experiments}

Our goal is to lay the groundwork for competitive DD-based solvers by ensuring efficient DD construction.
State-of-the-art solvers leverage countless techniques; we focus purely on what DDs achieve alone.
We use naive DD-based branch-and-bound \cite{BnB,BnB2} without initial presolve; the only presolve at each node is Feasibility-Based Bound Tightening \cite{achterberg2018}.
This minimal setup isolates DD construction cost, directly validating our theoretical analysis.
To minimize overhead, our implementation preallocates and reuses DD memory.
Experiments ran on AMD EPYC 9655 nodes (2.7 GHz) with up to 300 GB RAM on a single-thread. 

\subsection{Width Scaling}
\label{sec:width_scaling}

To validate linear scaling in \width, we use random binary knapsack instances with $K = 2$ and $|s| = m = 1$, isolating the \width\ dependence and enabling scaling to $\width = 10^8$ within memory limits.
Instances are intentionally hard ($n = 100$, capacity at 50\% of expected weight); with a 10-minute time limit, none are solved, maximizing the number of full-width layer constructions. 

Figure~\ref{fig:width_scaling} shows per-layer construction time, averaged across all layers and 10 instances, for widths from $10^3$ to $10^8$.
All subroutines scale linearly with \width, confirming $\mathcal{O}(\width)$ per-layer complexity.
Linear regression yields per-layer runtimes of $21.7 \pm 0.6$ ms per $10^6$ $\width$ for restricted DDs ($R^2 = 0.988$) and $34.7 \pm 0.6$ ms per $10^6$ $\width$ for relaxed DDs ($R^2 = 0.994$).
At $\width = 10^8$, measured per-layer times are 2.2 s (restricted) and 3.6 s (relaxed).
Memory scales at $2.58$ GB per $10^6$ $\width$ ($R^2 > 0.999$); optimality gap improves with larger widths (Appendix~\ref{appendix:width_scaling_figures}).
Memory, not time, was the bottleneck for scaling these experiments further. 

\begin{figure}[h!t]
\centering
\begin{subfigure}[b]{0.49\textwidth}
\centering
\includegraphics[width=\textwidth]{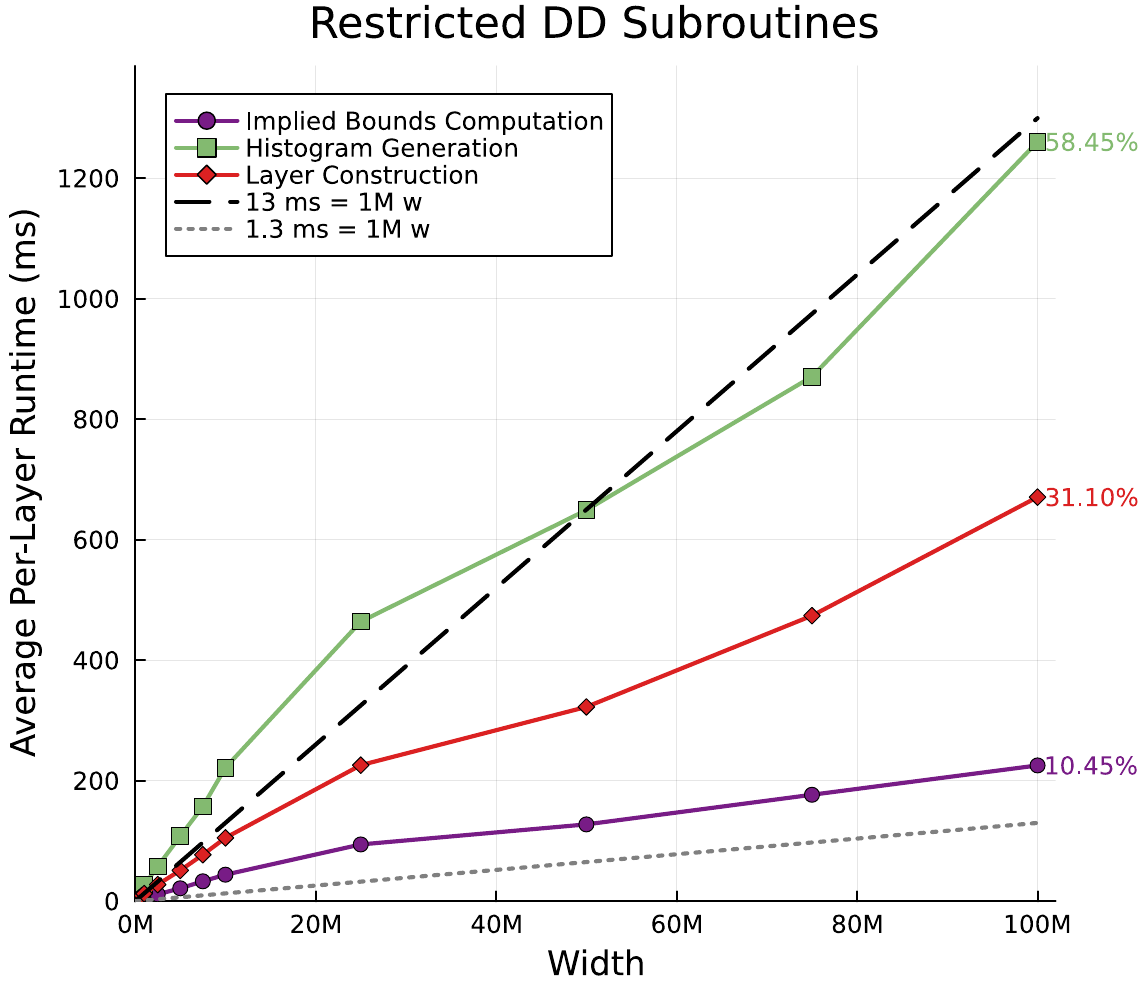}
\caption{Restricted DD subroutines}
\label{fig:restricted_runtime}
\end{subfigure}
\hfill
\begin{subfigure}[b]{0.49\textwidth}
\centering
\includegraphics[width=\textwidth]{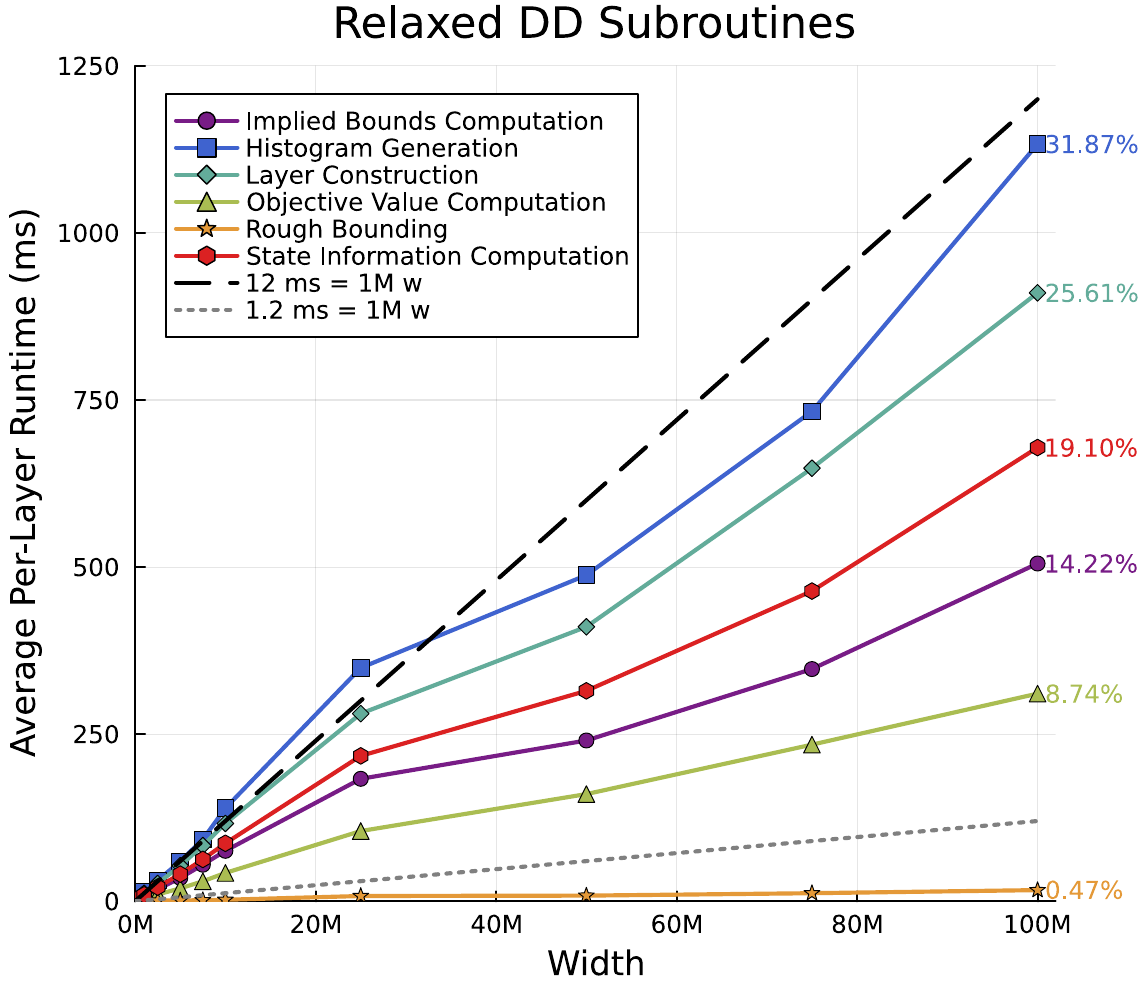}
\caption{Relaxed DD subroutines}
\label{fig:relaxed_runtime}
\end{subfigure}
\caption{Per-layer runtime vs.\ width on hard binary knapsack instances.}
\label{fig:width_scaling}
\end{figure}

\subsection{Domain and State Scaling}
\label{sec:ks_scaling}

To validate the $K$ (domain size) and $|s|$ (state size) components, we use multi-constraint generalized knapsack instances with $\width = 10^4$, $n = 100$, and scaling factor $f$ from 1 to 1000 ($K = 2f$, $|s| = f$).

Figure~\ref{fig:ks_scaling} shows per-layer construction time, averaged across all layers and 10 instances per configuration.
For restricted DDs, all subroutines scale linearly with $f$, confirming $\mathcal{O}(K + |s|)$ dependence; linear regression yields $3.87 \pm 0.07$ ms per unit $f$ ($R^2 = 0.989$).
For relaxed DDs, all subroutines except state computation scale linearly, fitting $0.31$ ms per unit $f$; state computation is $\mathcal{O}(K \cdot |s|)$ (Theorem~\ref{thm:tightness}), fitting $0.0025$ ms per $f^2$ (combined $R^2 = 0.986$).
State computation dominates at large $f$, consistent with $\mathcal{O}(K \cdot \width \cdot |s|)$ total complexity.
Memory remains modest (${\sim}1$ GB at $f = 1000$) because allocation depends on $\width \cdot |s|$, not $K$, and our implementation double-buffers state information ($2 \cdot \width \cdot |s|$ storage rather than $n \cdot \width \cdot |s|$).
Memory scales at $0.98 \pm 0.02$ MB per unit $f$ ($R^2 = 0.981$).

\begin{figure}[h!t]
\centering
\begin{subfigure}[b]{0.49\textwidth}
\centering
\includegraphics[width=\textwidth]{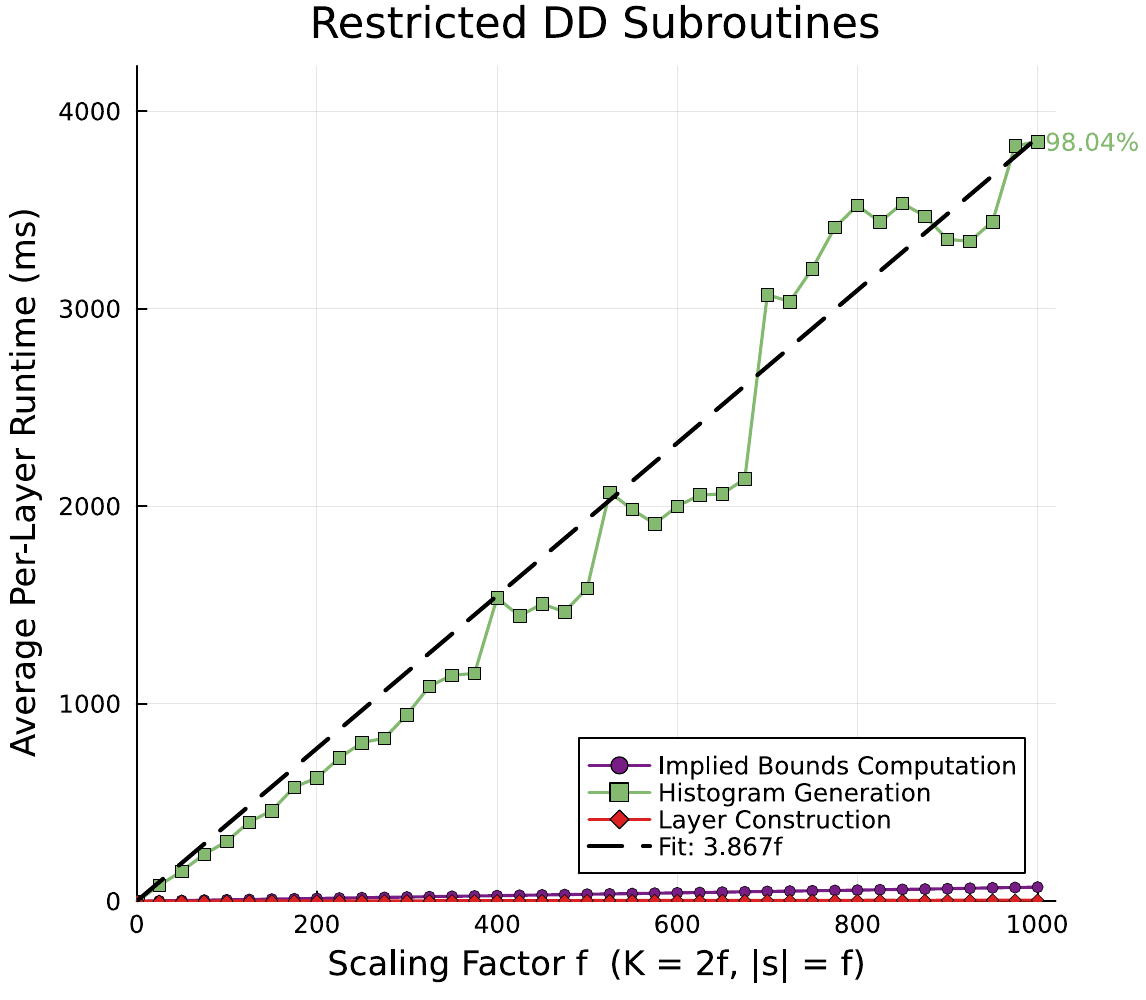}
\caption{Restricted DD subroutines}
\label{fig:ks_restricted_runtime}
\end{subfigure}
\hfill
\begin{subfigure}[b]{0.49\textwidth}
\centering
\includegraphics[width=\textwidth]{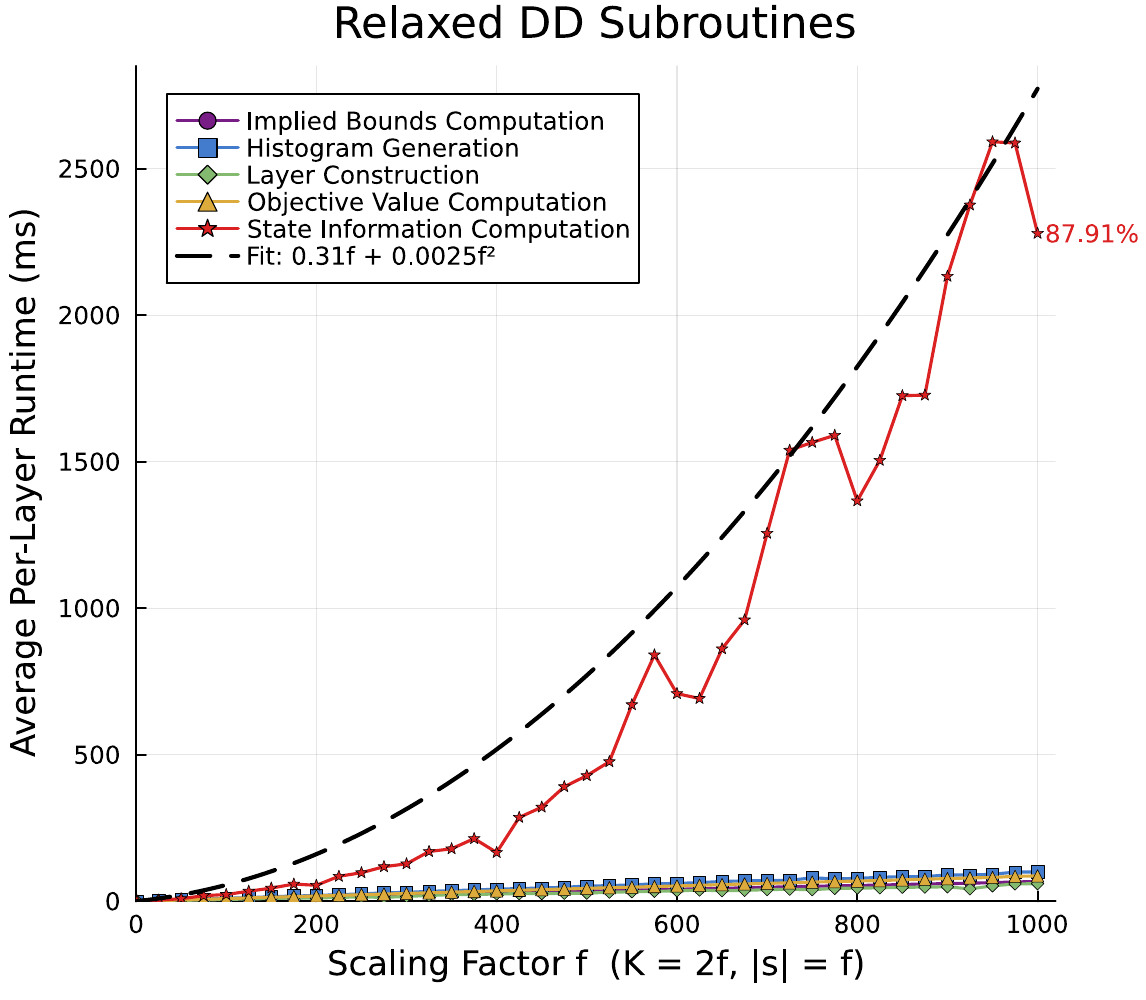}
\caption{Relaxed DD subroutines}
\label{fig:ks_relaxed_runtime}
\end{subfigure}
\caption{Per-layer runtime vs.\ scaling factor $f$ ($K = 2f$, $|s| = f$) at fixed $\width = 10^4$.}
\label{fig:ks_scaling}
\end{figure}

\subsection{Solver Comparison}
\label{sec:solver_comparison}

We now compare our solver against established MIP solvers on Subset Sum, one of Karp's original NP-Complete problems \cite{Karp1972}.
Each instance has $n$ binary variables, $m = 1$ equality constraint with coefficients sampled uniformly from $[1, 2^b]$, and a hidden feasible solution guaranteeing satisfiability.
ImplicitDDs uses $\width = 10^5$ and all solvers receive a 15-minute time limit.

We test $n \in \{20, 25, 30, 35, 40\}$ and $b \in \{16, 20, 24, 28, 32\}$ bits per coefficient, with 40 random seeds per configuration (1000 instances total).
Each instance was given to ImplicitDDs (default), ImplicitDDs (reordered) with variables sorted by coefficient magnitude, Gurobi, HiGHS, and SCIP.
Figure~\ref{fig:solver_comparison} shows performance profiles.
SCIP times out on 40\% of instances; HiGHS solves all but one.
The other 3 solvers all finish within 1 minute.
Gurobi had a geometric mean of 4.99s and a max of 55.2s.
ImplicitDDs (default) had a geometric mean of 2.4s, a max of 8.4s, and a median 2.2x speedup over Gurobi.
ImplicitDDs (reordered) had a geometric mean of 0.12s, a max of 1.6s, and a median 45.5x speedup over Gurobi.

ImplicitDDs is a research prototype rather than a production solver, and omits queue management heuristics that would be necessary for general-purpose benchmarking.
Subset Sum is a suitable testbed because it yields hard problems with few variables.
Queues remain small and construction speed is the dominant factor, making the benchmark a direct test of our core contribution.
\begin{figure}[h!t]
\centering
\includegraphics[width=0.7\textwidth]{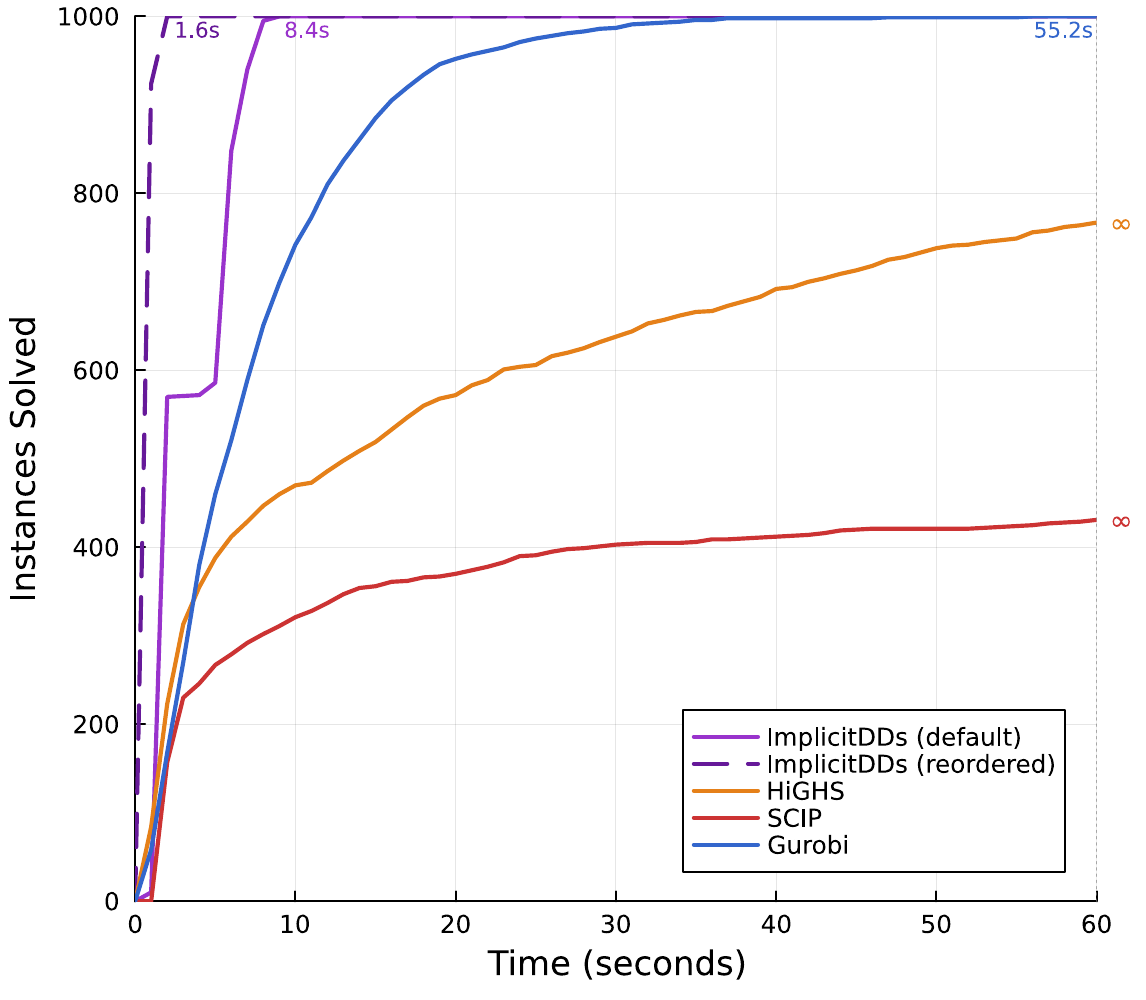}
\caption{Performance profile on 1000 Subset Sum instances.}
\label{fig:solver_comparison}
\end{figure}

\section{Conclusion}
\label{sec:conclusion}

Standard Decision Diagrams have a fundamental bottleneck: $\mathcal{O}(\width \log \width)$ for restricted DDs, $\mathcal{O}(\width^2)$ for relaxed.
Implicit DDs eliminate those bottlenecks by reducing both to $\mathcal{O}(\width)$, the theoretical limit.
On Subset Sum, our implementation yields a 45.5x median speedup over Gurobi.

Yet construction is merely one component of a DD-based solver.
MIP technology has matured through decades of refinement in presolve, branching, cutting planes, decomposition, and LP solving; DD-based optimization is only beginning this journey.
A problem-specific variable ordering contributed 20x here; the DD literature offers general alternatives we did not explore. 
Many MIP staples lack DD-specific counterparts but offer promising research directions. 
Among them: presolve, decomposition, and custom LP solvers.
The gap between a fast subroutine and a broadly competitive solver is wide; this work begins to close it.

\paragraph{Acknowledgments.} This work was funded by the Canada Research Chair on Healthcare Analytics and Logistics (CRC-2021-00556).

%
%
%
\bibliographystyle{splncs04}
\bibliography{bibliography}

\FloatBarrier
\appendix

\section{Relaxed Decision Diagram by Top Down Merging}
\label{appendix:relaxed_dd_methods}

We analyze top-down merging and compare to separation (Section \ref{sec:relaxed_dd}).

\subsection{Top-Down Construction via Merging}
\label{appendix:topdown_relaxed}

This differs from restricted DD construction (Algorithm \ref{algo:topdown_restricted}) only in the method of width control: nodes are \emph{merged} rather than removed.
One new operation: \textsc{SelectNodesToMerge}$(Q)$ selects two nodes.
Algorithm \ref{algo:topdown_relaxed} formalizes this.

\begin{algorithm}[h!t]
\SetAlgoLined
    \textbf{Input:} The root node $r$ for a given \Pgen. Problem-specific functions: $\textsc{UpdateState}$, $\textsc{IsFeasible}$, $\textsc{SelectNodesToMerge}$, and $\textsc{MergeStates}$. \\
    Initialize $Q \leftarrow \{r\}$, $Q_{next} \leftarrow \emptyset$\\

    \While{$Q \neq \emptyset$}{
        $Q_{next} \leftarrow \emptyset$\\
        \tcp{1. Child generation}
        \ForEach{node $u \in Q$}{
            \ForEach{label $l \in d(u)$}{
                Create a new node $v$: $Q_{next} \leftarrow Q_{next} \cup \{v\}$\\
                Add arc $a_{uv}$, labeled $l$, from $u$ to $v$; $s(v) = \textup{\textsc{UpdateState}}(s(u), x_{\ell(u)+1}, l)$\\
            }
        }
        \tcp{2. Feasibility pruning}
        \ForEach{node $u \in Q_{next}$}{
            \If{\textup{\textsc{IsFeasible}}($s(u)$) $=$ false}{
                $Q_{next} \leftarrow Q_{next} \backslash \{u\}$\\
            }
        }
        \tcp{3. Width limiting via merging}
        \While{$|Q_{next}| > \width$}{
            $(u, v) \leftarrow$ \textsc{SelectNodesToMerge}$(Q_{next})$\\
            Create a new node $\beta$: $Q_{next} \leftarrow Q_{next} \backslash \{u,v\} \cup \{\beta\}$\\
            Redirect in-arcs of $u$ and $v$ to $\beta$\\
            $s(\beta) \leftarrow$ \textsc{MergeStates}$(u, v)$\\
        }
        \tcp{4. Additional operations}
        \textbf{(Optional)} Apply additional operations (e.g., rough bounding \cite{ddo})\\

        $Q \leftarrow Q_{next}$\\ 
    }

    \Return{\relDD}
    \caption{Top-Down Relaxed Decision Diagrams \cite{DDforO}}
    \label{algo:topdown_relaxed}
\end{algorithm}

\subsubsection{Complexity Analysis}

We extract the inherent costs to compare with Algorithm \ref{algo:separation_relaxed}.
The algorithm performs three operations per layer: (1) child generation, (2) feasibility pruning, and (3) width limiting via merging.

\textbf{Child generation:} Identical to Algorithm \ref{algo:topdown_restricted}.
Per layer: \[
\textcolor{safeA1}{\mathcal{O}(K \cdot \width \cdot |s|)} + \mathcal{O}(K \cdot \width) \cdot C_{updt}
\]

\textbf{Feasibility pruning:} Identical to Algorithm \ref{algo:topdown_restricted}.
Per layer: \[
\textcolor{safeA1}{\mathcal{O}(K \cdot \width \cdot |s|)} + \mathcal{O}(K \cdot \width) \cdot C_{isf}
\]

\textbf{Width limiting via merging:}
Reduces $|Q_{next}|$ from up to $K \cdot \width$ nodes to $\width$ nodes via repeated pairwise merging $\big(\mathcal{O}(K \cdot \width)$ merges$\big)$.
Each merge:
(1) Calls $\textsc{SelectNodesToMerge}$ ($C_{selmrg}$).
(2) Redirects in-arcs from nodes $u$ and $v$ to merged node $\beta$.
In the worst case (growing chain), successive merges handle 2, 3, ..., up to $K \cdot \width - \width + 1$ arcs, $\mathcal{O}(K \cdot \width)$ on average.
(3) Calls $\textsc{MergeStates}$, $\big(\mathcal{O}(|s|) + C_{mrg}\big)$.
Per layer: \[
\textcolor{safeA1}{\mathcal{O}(K^2 \cdot \width^2 + K \cdot \width \cdot |s|)} + \mathcal{O}(K \cdot \width) \cdot (C_{selmrg} + C_{mrg})
\]

The $\mathcal{O}(K^2 \cdot \width^2)$ factor arises from $\mathcal{O}(K \cdot \width)$ merges, each redirecting $\mathcal{O}(K \cdot \width)$ in-arcs on average. For $n$ layers:
\[
\textcolor{safeA1}{\mathcal{O}(n \cdot (K^2 \cdot \width^2 + K \cdot \width \cdot |s|))} + \mathcal{O}(n \cdot K \cdot \width) \cdot (C_{selmrg} + C_{mrg} + C_{updt} + C_{isf})
\]

\subsubsection{Application to Integer Programs}
For \Pint, $|s| = m$.
$C_{updt} = \mathcal{O}(m)$ (computing residuals via Equation \ref{eq:residual_update});
$C_{mrg} = \mathcal{O}(m)$ (element-wise maximum of residuals);
$C_{isf} = \mathcal{O}(m)$ (checking residuals);
$C_{selmrg}$: selecting from $K \cdot \width$ candidates contributes $\mathcal{O}(K \cdot \width)$ per layer.
So the overall complexity becomes:
\[
\textcolor{safeA1}{\mathcal{O}(n \cdot (K^2 \cdot \width^2 + K \cdot \width \cdot m))} + \mathcal{O}(n \cdot K^2 \cdot \width^2) + \mathcal{O}(n \cdot K \cdot \width \cdot m)
\]
All problem-specific terms are subsumed by inherent costs.

\subsection{Comparison}
Comparing inherent costs, top-down (left) vs. separation (right):
\begin{align*}
\textcolor{safeA1}{\mathcal{O}\Big(n \cdot (K^2 \cdot \width^2 + K \cdot \width \cdot |s|)\Big)} &\quad \text{?} \quad \textcolor{safeA1}{\mathcal{O}\Big(n \cdot K \cdot \width^2 \cdot |s|\Big)} \\
\textcolor{safeA1}{\mathcal{O}\Big((n \cdot K \cdot \width) \cdot (K \cdot \width + |s|)\Big)} &\quad \text{?} \quad \textcolor{safeA1}{\mathcal{O}\Big((n \cdot K \cdot \width) \cdot (\width \cdot |s|)\Big)} \\
\textcolor{safeA1}{\mathcal{O}\Big(K \cdot \width + |s|\Big)} &\quad \text{?} \quad \textcolor{safeA1}{\mathcal{O}\Big(\width \cdot |s|\Big)}
\end{align*}

Top-down has an extra $K$ factor; separation has an extra $|s|$ factor.
Top-down is faster when $K < |s|$, and vice versa.
Since DD solvers favor smaller $K$, we expect $K < |s|$ for competitive instances, slightly favoring top-down.

\section{Additional Width Scaling Figures}
\label{appendix:width_scaling_figures}

Figure~\ref{fig:memory} shows preallocated DD memory scaling linearly with $\width$.
Figure~\ref{fig:gap_over_time} shows optimality gap evolution over time for various DD widths.

\begin{figure}[h!t]
\centering
\begin{subfigure}[b]{0.49\textwidth}
\centering
\includegraphics[width=\textwidth]{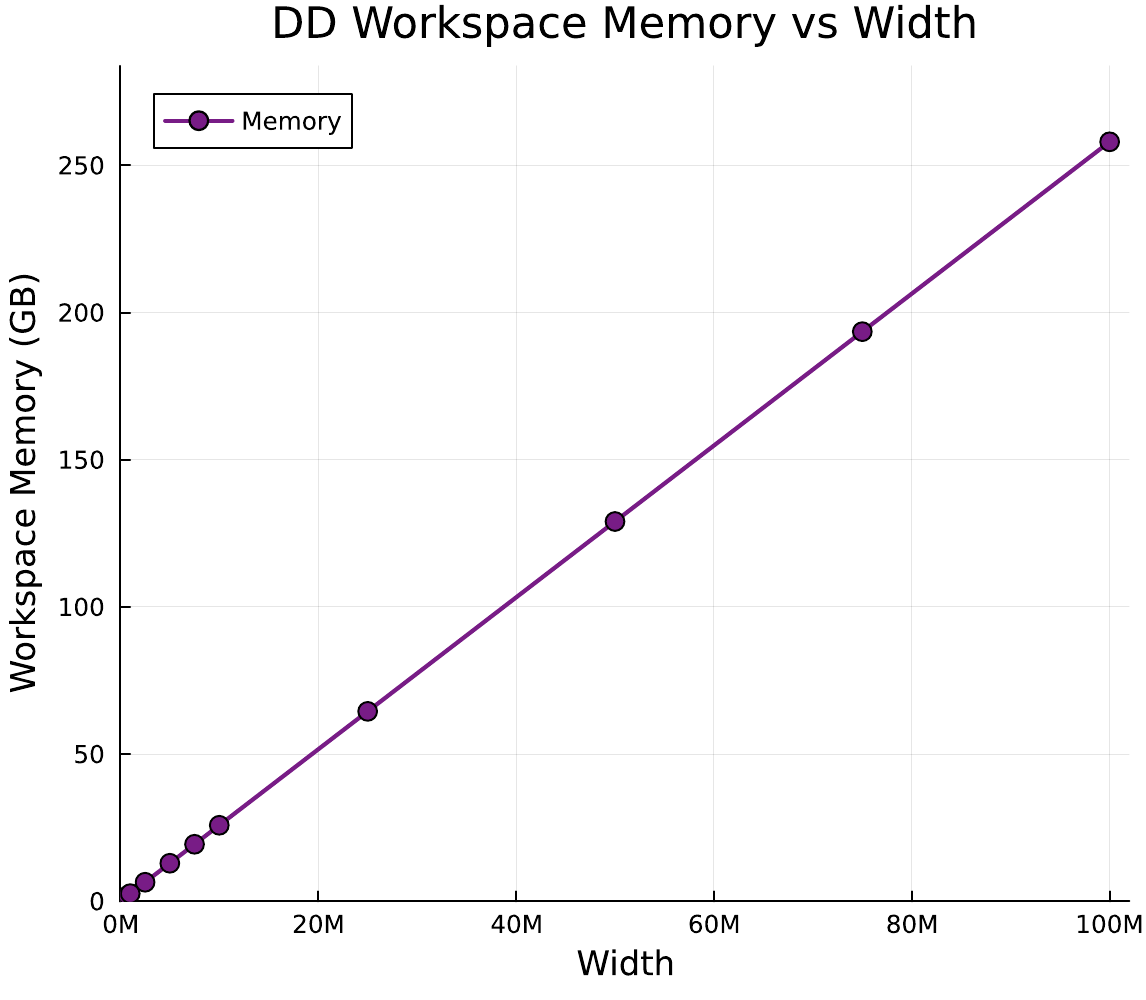}
\caption{Memory vs.\ width}
\label{fig:memory}
\end{subfigure}
\hfill
\begin{subfigure}[b]{0.49\textwidth}
\centering
\includegraphics[width=\textwidth]{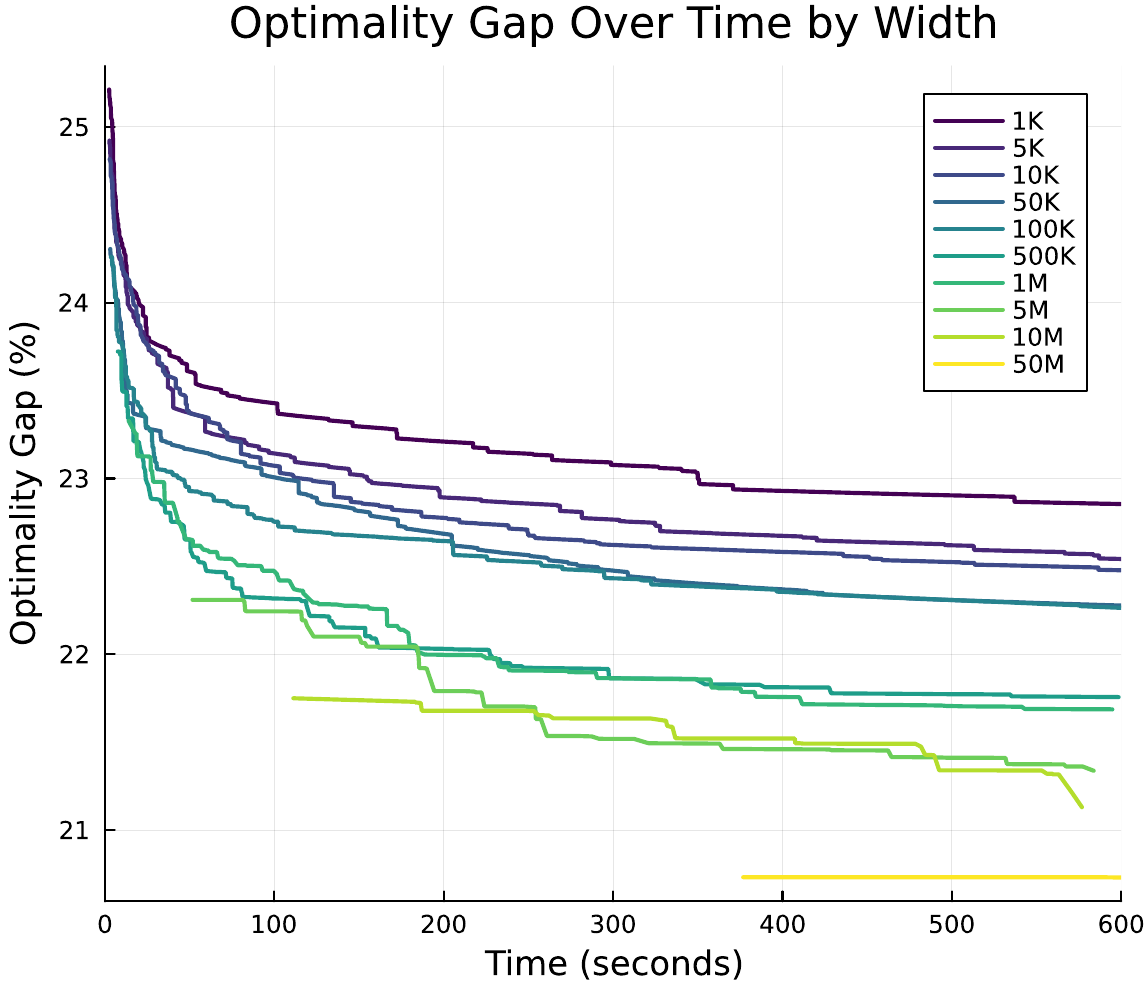}
\caption{Optimality gap over time}
\label{fig:gap_over_time}
\end{subfigure}
\caption{Additional width scaling results on hard knapsack instances.}
\label{fig:appendix_width_scaling}
\end{figure}

\end{document}